\documentclass[11pt,letterpaper]{amsart}
\usepackage[utf8]{inputenc}
\usepackage[T1]{fontenc}
\usepackage[english]{babel}
\usepackage{subfigure}
\usepackage{amsmath}
\usepackage{amsthm} 
\usepackage{framed}
\usepackage[normalem]{ulem}
\usepackage{hyperref} 
\usepackage[OT2,OT1]{fontenc}
\usepackage{amsfonts}
\usepackage{tikz-cd}
\usepackage{tikz}
\usepackage{amssymb}
\usepackage{csquotes}
\usepackage{calrsfs} 
\usepackage[
backend=biber,
style=alphabetic,
sorting=nty,doi=false,isbn=false, url=false
]{biblatex}
\usepackage{accents}
\usepackage{adjustbox}
\usepackage{xcolor}
\definecolor{darkgreen}{RGB}{0,100,0}
\definecolor{DARKGREEN}{RGB}{0,100,0}

\definecolor{darkblue}{RGB}{0,0,139}
\definecolor{DARKBLUE}{RGB}{0,0,139}

\newlength{\dhatheight}

\makeatletter
\newcommand{\xleftrightarrow}[2][]{\ext@arrow 3359\leftrightarrowfill@{#1}{#2}}
\makeatother

\newcommand{\xdasharrow}[2][->]{
\tikz[baseline=-\the\dimexpr\fontdimen22\textfont2\relax]{
\node[anchor=south,font=\scriptsize, inner ysep=1.5pt,outer xsep=2.2pt](x){#2};
\draw[shorten <=3.4pt,shorten >=3.4pt,dashed,#1](x.south west)--(x.south east);
}
}
\theoremstyle{plain}
\newtheorem*{theorem*}{Theorem}
\newtheorem{theorem}{Theorem}[section]
\newtheorem{prop}[theorem]{Proposition}
\newtheorem{defi}[theorem]{Definition}
\newtheorem{lemma}[theorem]{Lemma}
\theoremstyle{definition}
\newtheorem{example}[theorem]{Example}

\newtheorem*{corollary*}{Corollary}
\newtheorem*{prop*}{Proposition}

\newtheoremstyle{named}{}{}{\itshape}{}{\bfseries}{.}{.5em}{\thmnote{#3's }#1}
\theoremstyle{named}

\newcommand{\C}{\mathbb{C}}
\newcommand{\Z}{\mathbb{Z}}

\newcommand{\Q}{\mathbb{Q}}
\newcommand{\A}{\mathbb{A}}
\newcommand{\N}{\mathbb{N}}

\newcommand{\et}{\textup{\'et}}

\newcommand{\alg}{\operatorname{alg}}

\newcommand{\CH}{\operatorname{CH}}
\newcommand{\tors}{\operatorname{tors}}
\theoremstyle{remark}
\newtheorem{remark}[theorem]{Remark}
\tikzcdset{
  cells={font=\everymath\expandafter{\the\everymath\displaystyle}},
}

\newcommand{\info}{{
  \bigskip
  \footnotesize
  \textsc{Departamento de Matemática, Universidad Técnica Federico Santa María, Avenida España 1680, Valparaíso, Chile}\par\nopagebreak
  \textit{E-mail address}: \texttt{pedro.montero@usm.cl }\par\nopagebreak
\vspace{2mm}
  \textsc{Université de Strasbourg, Institut de recherche mathématique avancée (IRMA), France}\par\nopagebreak
  \textit{E-mail address}: \texttt{ivan.rosas-soto@math.unistra.fr}
  }}
  

\author{Pedro Montero, Iv\'an Rosas-Soto}
\date{November 2026}
\title[Algebraic Cycles on Some Fano Varieties. Part I]{Algebraic cycles on some Fano varieties. Part~I: Quartic double fivefolds}
\begin{document}
\maketitle

\begin{abstract}
We study algebraic cycles on quartic double fivefolds, which are Fano fivefolds whose middle cohomology is numerically similar to that of a Calabi-Yau threefold, continuing the work of Skauli on lines, $1$-cycles and the coniveau filtration on double covers of projective space. By means of cubic surfaces which are totally ramified along the branch locus, we prove the integral Hodge conjecture for all of them. We also show that the comparison maps between their Chow groups and étale Chow groups are isomorphisms except possibly in codimension three, where we describe the kernel and the cokernel in terms of Bloch-Ogus cohomology and, via a Sarkisov link, of the unramified cohomology of a cubic surface over the function field of $\mathbb{P}^3$.
\end{abstract}


\section*{Conventions and notations}
Throughout this paper we will assume that the base field $k$ is of $\text{char}(k)=0$, where in some subsections we will restrict to $k=\C$. For an abelian group $G$ we denote $G_{\tors}$ the subgroup of torsion elements of $G$. For a commutative ring $\A$ we denote $G_{\A}:=G\otimes_\Z \A$ and for an element $g \in G$ we denote $g_{\A}:= g\otimes 1 \in G_{\A}$. 

\section{Introduction}

Chow groups provide important information about the algebro-geometric structure of a variety, which can be used to address different problems in algebraic geometry, such as rationality questions, enumerative problems, and the behaviour of such properties in families, among others.

One of the open questions about algebraic cycles in a complex smooth projective variety $X$ is the so-called \emph{Hodge conjecture}, which relates the algebro-geometric information of the variety with its analytic counterpart. We refer the interested reader to the survey book \cite{Lew} for a comprehensive introduction to the subject. Originally, the version of the conjecture with integral coefficients was stated as follows: for all smooth projective varieties $X$ over $\C$ and for all $0 \leq i \leq \operatorname{dim}(X)$, the cycle class map
\[
    c^i:\CH^i(X) \to \operatorname{H}^{2i}_B(X,\Z)
\]
surjects onto the integral Hodge classes $\operatorname{Hdg}^{2i}(X,\Z)$. In this form, Atiyah and Hirzebruch \cite{AH} provided a counterexample in 1962, consisting of a torsion class in $H^4_B(X,\Z)$, for a Godeaux-Serre variety $X$ of dimension $7$, which is not algebraic. The situation does not improve if one considers only the torsion-free part of Hodge classes of a smooth projective variety: Kollár, in \cite{BCC}, produced the so-called \emph{Trento examples}, obtained by considering very general hypersurfaces of dimension $4$ and degree $125$, where the generator class of $\operatorname{H}^4_B(X,\Z)$ is not algebraic.

Despite this negative evidence for the integral version of the Hodge conjecture, its rational counterpart remains open to this day: for all smooth projective varieties $X$ over $\C$ and for all $0 \leq i \leq \operatorname{dim}(X)$ the cycle class map
\[
    c_\Q^i:\CH^i(X)_\Q \to \operatorname{Hdg}^{2i}(X,\Q)
\]
is surjective.

As we have mentioned, algebraic cycles and topological information also play a role in the study of rationality (or birational) problems. The first birational invariant of smooth projective varieties to recall is the Artin-Mumford invariant, given by $\operatorname{H}^3_B(X,\Z)_{\operatorname{tors}}$, which encodes the topological information of the Brauer group of $X$, in other words, it provides information about the Severi-Brauer varieties of $X$. Another invariant of topological nature is the obstruction to the integral Hodge conjecture in codimension $2$, i.e.\ $(\operatorname{Hdg}^4(X,\Z)/\operatorname{Im}(c^2))_{\operatorname{tors}}$. Similarly, for curve classes, one has that $\operatorname{Hdg}^{2n-2}(X,\Z)/\operatorname{Im}(c^{n-1})$ is a birational invariant of smooth projective varieties. All of these invariants are related to another broad family of birational invariants, the so-called \emph{unramified cohomology groups} introduced by Colliot-Thélène and Ojanguren in \cite{CTO89}, which have been successfully used to measure the failure of the Hodge conjecture for degree four integral classes \cite{CTV} and to produce new counterexamples to rationality questions \cite{CTP16} by means of the decomposition of the diagonal (see \S \ref{section:Chow}). We refer the interested reader to the survey articles \cite{Pey19,Voi19,Sch21} for an introduction to the subject.

\vspace{2mm}

In this article we study algebraic cycles on some Fano varieties whose middle cohomology is numerically similar to that of a Calabi-Yau threefold in the sense of Iliev and Manivel \cite{IM15}, both classically and from the point of view of the étale topology, in the following sense.

\vspace{2mm}

We can obtain Chow groups by taking the hypercohomology of Bloch's complexes. More precisely, for each $j$, the presheaf $z^i(-,j)\colon U \mapsto z^i(U,j)$ associated with Bloch's complex is a sheaf for all topologies $t \in \left\{\et,\ \operatorname{Nis},\ \operatorname{Zar} \right\}$ (see \cite{Ge04}), and hence $z^i(-,\bullet)$ is a complex of sheaves on the small étale, Nisnevich, and Zariski sites of $X$. We define the complex of $t$-sheaves
\[
\Z_X(i)_t := z^i(-,\bullet)_t [-2i].
\]
In this case, by taking $t = \operatorname{Zar}$, the hypercohomology of the complex recovers Chow groups:
\[
\CH^i(X) \simeq \mathbb{H}^{2i}_{\operatorname{Zar}}(X, \Z(i)).
\]
From the perspective of triangulated categories, we can recover Chow groups by using the category of Nisnevich motives as
\[
\CH^i(X)\simeq \operatorname{Hom}_{\operatorname{DM}(k,\Z)}(M(X),\mathbf{1}(i)[2i]).
\]

All of these isomorphisms are related to the Zariski or Nisnevich topology. As we have mentioned, we may also proceed by considering the étale topology and obtain an étale version of Chow groups. Depending on the reference, this is called \emph{Lichtenbaum cohomology} or \emph{étale motivic cohomology}, which we denote by $\CH^i_\et(X)$. This new approach provides additional tools to study problems about algebraic cycles on smooth projective varieties, in particular the Hodge conjecture. For a smooth projective variety over a field of characteristic zero, we define
\[
    \CH^i_\et(X):=\operatorname{Hom}_{\operatorname{DM}_\et(k,\Z)}(M_\et(X),\Z(i)[2i])\simeq \mathbb{H}^{2i}_\et(X,\Z(i)).
\]

A first important feature (see \cite{Kahn,CTV}) of the étale Chow groups is that they fit into the following commutative diagram
\[
\begin{tikzcd}
    0 \arrow[r]& \CH^2(X) \arrow[r,"\kappa^2"]\arrow[d]& \CH^2_\et(X) \arrow[d] \arrow[r] & \operatorname{H}_{\operatorname{nr}}^3(X,\Q/\Z(2))\arrow[d] \arrow[r] & 0 \\
    0 \arrow[r] & \operatorname{Im}(c^2)\arrow[r] & \operatorname{H}^4_B(X,\Z)\arrow[r] & \operatorname{H}^4_B(X,\Z)/\operatorname{Im}(c^2)\arrow[r]& 0
\end{tikzcd}
\]
thus measuring the failure of the integral Hodge conjecture in degree $4$.

More generally, one of the fundamental results that illustrate the relation between algebraic cycles and étale motivic cohomology was given by Rosenschon and Srinivas in \cite{RS}, concerning the Hodge and Tate conjectures. In brief, for a smooth projective variety $X$ over $\C$, and for fixed $i \in \N$, there exists a cycle class map $c_\et^i:\CH_\et^i(X) \to \operatorname{H}_B^{2i}(X,\Z)$, whose image is contained in $\operatorname{Hdg}^{2i}(X,\Z)$. Therefore, we can formulate an étale analogue of the Hodge conjecture, namely: to ask whether the map $c_\et^i:\CH_\et^i(X) \to \operatorname{Hdg}^{2i}(X,\Z)$ is surjective. The first point to note is that, in contrast with the usual cycle class map, on torsion classes we have a surjection
\[
    c_\et^i\Big|_{\operatorname{tors}}:\CH_\et^i(X)_{\operatorname{tors}} \twoheadrightarrow \operatorname{Hdg}^{2i}(X,\Z)_{\operatorname{tors}}.
\]
This follows from the isomorphism between étale motivic cohomology groups with finite coefficients and étale cohomology, together with the comparison between étale and Betti cohomology with finite coefficients, see \cite{Mil80}. Consequently, we obtain the equivalence
\[
    c^i_\Q:\CH^i(X)_\Q \twoheadrightarrow \operatorname{Hdg}^{2i}(X,\Q) \iff c_\et^i:\CH_\et^i(X) \twoheadrightarrow \operatorname{Hdg}^{2i}(X,\Z).
\]

Since $c^i=c^i_\et\circ\kappa^i$, where $\kappa^i\colon\CH^i(X)\to\CH^i_\et(X)$ denotes the canonical comparison map associated with the change of sites (see \S\ref{section:Chow etale}), the integral Hodge conjecture for a variety satisfying the rational Hodge conjecture is thus controlled by the comparison maps $\kappa^i$, which are isomorphisms after tensoring with $\Q$.

\vspace{2mm}

From the perspective of $0$-cycles, an important class of algebraic varieties is given by Fano varieties (i.e.\ smooth projective varieties with ample anticanonical divisor). These form a concrete class of rationally connected varieties, meaning that all of their $0$-cycles are rationally equivalent, and they arise naturally in the birational classification of higher-dimensional algebraic varieties (see, for instance, \cite{Mur04} for a survey on algebraic cycles on Fano varieties). In regard of these considerations, in this article we focus on a first example of Fano variety whose middle cohomology is numerically similar to that of a Calabi-Yau threefold. The case of Fano varieties with Hodge structures of level one is treated in the companion article \cite{MRII}.

\vspace{2mm}

The Hodge structures of odd weight that we will consider are those of \emph{Calabi--Yau type}, which look like the middle cohomology of a Calabi--Yau threefold, and Fano manifolds with such a middle cohomology have been studied by Iliev and Manivel \cite{IM15} (see also \cite{FIK14}). One of such candidates (see \cite[Remark 1.4]{Abu20}) is the \emph{quartic double fivefold}, i.e.\ a double cover $f\colon X\to\mathbb{P}^5$ branched along a smooth quartic hypersurface $B\subset \mathbb{P}^5$: it is a del Pezzo fivefold of degree two, and $h^{5,0}(X)=0$, $h^{4,1}(X)=1$, $h^{3,2}(X)=90$. In particular, homologically trivial cycles of codimension three are far from being controlled by an abelian variety: for general $X$, the group $\operatorname{Griff}^3(X)\otimes\Q$ is infinitely generated \cite{FIK14}. Nevertheless, our main result is the following.

\begin{theorem}[{Theorem~\ref{thm:Hodge}}]\label{thm:main thm}
Let $f:X\to \mathbb{P}^5$ be a smooth double cover branched along a smooth quartic hypersurface $B\in |\mathcal{O}_{\mathbb{P}^5}(4)|$. Then the integral Hodge conjecture holds for $X$.
\end{theorem}

The rational Hodge conjecture for $X$ is straightforward, and for the integral version the main difficulty lies in degree six: $\operatorname{H}^6(X,\Z)$ is generated by $\beta=\frac{1}{2}f^\ast H^3$, whereas the preimages of planes of $\mathbb{P}^5$ have class $2\beta$ and, if $B$ is very general, the surfaces contained in the ramification divisor $R\simeq B$ have even classes. In order to produce odd degree classes, we will introduce an auxiliary geometric construction: we say that a cubic surface $T$ contained in a linear subspace $\mathbb{P}^3\subset\mathbb{P}^5$ is \emph{totally ramified} along $B=\{q=0\}$ if $q|_T$ is the square of a nonzero section of $\mathcal{O}_T(2)$, i.e.\ if $B\cap T=2C$ for a curve $C\subset T$, and such a surface lifts isomorphically to $X$, with class $3\beta$. We show that the incidence variety of totally ramified cubic surfaces dominates the space of quartics, with general fibre of dimension six (Theorem~\ref{thm:dominance}), and, by a properness argument, that every smooth quartic fourfold admits totally ramified cubic surfaces (Theorem~\ref{thm:every-B}).

In this situation, the rational Hodge conjecture follows from the double covering structure, which allows one to lift rational classes of linear cycles from projective space. From this perspective, one may ask whether such cycles admit an interpretation as linear subspaces inside $X$, in analogy with the classical study of linear subspaces in complete intersections in projective space (cf. \cite{DM98}). This approach was first adopted by Tihomirov \cite{Tih}, who studied lines in the three-dimensional case, and later extended in the PhD thesis of Skauli \cite{ska}, who studied lines in double covers of arbitrary dimension, as well as Griffiths groups and the coniveau filtration in these varieties. From \cite{ska} we use the dimension of the Fano scheme of lines \cite[Proposition V.3.8]{ska} and the vanishing of the Griffiths group of $1$-cycles \cite[Theorem VI.1.6]{ska} (see also \cite[Theorem 2.3]{MP17}), and Appendix~\ref{appendixA} extends the construction of \cite[Chapter V]{ska} from lines to $f$-planes of dimension $r$. It is worth mentioning that a similar construction as the so-called totally ramified subvarieties considered in this article, was used in \cite[Proposition VII.3.5]{ska} in order to compare the two coniveau filtrations and in \cite[\S VII.4]{ska} to study double covers of $\mathbb{P}^4$. The difference is that in \cite{ska} the double cover is chosen after fixing the subvariety, whereas we show that every smooth quartic double fivefold contains $f$-planes and totally ramified cubic surfaces (Theorem~\ref{thm:every-B} and Remark~\ref{rmk:f-planes}).

Our second result concerns the comparison maps $\kappa^i$ for quartic double fivefolds. Recall that for a smooth complex variety $\mathcal{H}^q(\Z)$ denotes the Zariski sheaf associated with $U\mapsto \operatorname{H}^q_B(U,\Z)$.

\begin{theorem}[{Theorem~\ref{thm:kappa} and Proposition~\ref{prop:sarkisov}}]\label{thm:main-kappa}
Let $X$ be a quartic double fivefold as in Theorem~\ref{thm:main thm}. Then $\kappa^i$ is an isomorphism for $i\neq 3$, and moreover:
\begin{enumerate}
    \item $\operatorname{coker}(\kappa^3)\simeq \operatorname{H}^1(X,\mathcal{H}^4(\Z))\otimes\Q/\Z$, and $\ker(\kappa^3)$ is a quotient of $\operatorname{H}^4_{\operatorname{nr}}(X,\Q/\Z(3))\simeq\operatorname{H}^1(X,\mathcal{H}^4(\Z))_{\tors}$;
    \item projecting $X$ from an $f$-line, i.e.\ a curve of degree one, which is not contained in $R$, the function field of $X$ becomes the function field of an explicit smooth cubic surface $Y_K$ over $K=\C(\mathbb{P}^3)$, and $\operatorname{H}^4_{\operatorname{nr}}(X,\Q/\Z(3))$ embeds into $\operatorname{H}^4_{\operatorname{nr}}(K(Y_K)/K,\Q/\Z(3))$.
\end{enumerate}
In particular, if $\operatorname{H}^4_{\operatorname{nr}}(K(Y_K)/K,\Q/\Z(3))=0$, then $\kappa^3$ is an isomorphism if and only if $\operatorname{H}^1(X,\mathcal{H}^4(\Z))$ is a $\Q$-vector space.
\end{theorem}

The cubic surface $Y_K$ is obtained from the generic fibre of the projection of $X$ from an $f$-line not contained in $R$, a del Pezzo surface of degree two, by contracting a $(-1)$-curve defined over $K$. This is a classical Sarkisov link (cf. \cite{Isk96}), which was pointed out to us by Ivan Cheltsov. We also relate $\operatorname{coker}(\kappa^3)$ to Griffiths groups and to the refined unramified cohomology of Schreieder \cite{Sch23} (Remark~\ref{rmk:final}).

\vspace{2mm}

The article is structured as follows. In \S \ref{section:prelim} we recall the necessary preliminaries on Chow groups, intermediate Jacobians, decompositions of the diagonal and étale Chow groups. Finally, \S \ref{section:Other Fanos} is devoted to quartic double fivefolds, where we prove Theorem~\ref{thm:main thm} and Theorem~\ref{thm:main-kappa}.

We also include an Appendix, where by means of classical computations using incidence varieties, we obtain a description of certain Fano schemes of linear-like subvarieties on smooth double covers of projective space.

\section*{Acknowledgements}

The authors would like to thank the organisers of the conference ``Motivic Homotopy in Interaction'', held at CIRM Luminy from 4 to 8 November 2024. It was during this conference that this collaboration began, and we are especially grateful to \textsc{Fr\'ed\'eric D\'eglise}, \textsc{Adrien Dubouloz}, and \textsc{Johannes Nagel} for their constant support. We warmly thank \textsc{Stefan Schreieder} for all his comments and for pointing us to relevant references relating our results to his work on (refined) unramified cohomology groups. We are grateful to \textsc{Ivan Cheltsov} for pointing out to us the Sarkisov link of Proposition~\ref{prop:sarkisov}, and to \textsc{Lie Fu} for helpful discussions. We also thank the anonymous referees for the careful reading of previous versions of this article and for helpful suggestions that improved several arguments and the exposition. In particular, we are grateful to one of them for pointing out a gap in the construction of surfaces of odd degree on quartic double fivefolds in a previous version.

The first author was partially supported by ANID FONDECYT Regular grant 1231214 and by the ECOS-ANID Project ECOS230044. The second author was partially supported by the French ANR project “HQ-DIAG” (ANR-21-CE40-0015).

\section{Preliminaries}\label{section:prelim}

\subsection{Chow groups}\label{section:Chow}
Let $X$ be a smooth projective variety over a field $k$ of dimension $n$. We denote the Chow group of $X$ in codimension $i$ by $\CH^i(X)$, i.e., the group of closed subvarieties of $X$ of codimension $i$ modulo rational equivalence. We can grade Chow groups of $X$ by dimension of subvarieties, in which case we have $\CH^{i}(X):= \CH_{n-i}(X)$. 

We freely use the standard properties of Chow groups of smooth projective varieties, for which we refer to \cite{Fulton}: pull-backs $f^\ast\colon\CH^i(Y)\to \CH^i(X)$ along arbitrary morphisms $f\colon X\to Y$, and push-forwards $f_\ast[Z]=\deg(Z/f(Z))[f(Z)]$ along proper ones \cite[\S 1.4 and Chapter 8]{Fulton}, the intersection product, which makes $\CH^\ast(X)=\bigoplus_{i=0}^n\CH^i(X)$ a commutative graded ring compatible with pull-backs, and the projection formula $f_\ast(Z\cdot f^\ast Z')=f_\ast(Z)\cdot Z'$ \cite[Chapter 8]{Fulton}, the localization exact sequence $\CH_i(Z)\to \CH_i(X)\to \CH_i(X\setminus Z)\to 0$ for a closed subset $Z\subset X$ \cite[\S 1.8]{Fulton}, and the action of a correspondence $Z\in\CH^t(X\times Y)$ on cycles $T\in \CH^i(X)$, given by $Z(T)=(\operatorname{pr}_Y)_\ast(Z\cdot(T\times Y))\in \CH^{i+t-n}(Y)$ with $n=\dim(X)$ \cite[\S 16.1]{Fulton}. Let us recall from \cite[\S 10.3]{Fulton}:

\begin{defi}
    We say that a cycle $Z \in \CH^i(X)$ is \textup{algebraically trivial} if there is a smooth irreducible curve $C$ and $W \in \CH^i(C\times X)$ and two points $a$, $b \in C$ such that $W(a)=0$ and $W(b)=Z$. We denote by $\CH^i(X)_{\alg}$ the subgroup of algebraically trivial cycles of codimension $i$ on $X$.
\end{defi}

When $k=\C$ there is a well-defined cycle map $c^i_{\A}:\CH^i(X)_\A \to \operatorname{H}^{2i}_B(X,\A)$ to Betti cohomology for $\A=\Z, \ \Q$, whose image is a subset of the Hodge classes
\begin{align*}
    \text{Hdg}^{2i}(X,\A):=\left\{ \alpha \in \operatorname{H}^{2i}(X,\A) \ | \  \alpha_\C \in \operatorname{H}^{i,i}(X)\right\}.
\end{align*}
We denote $\CH^i(X)_{\hom}:= \text{ker}(c^i)$, where $c^i:=c^i_\mathbb{Z}$. It is clear that $\CH^i(X)_{\text{alg}} \subset \CH^i(X)_{\text{hom}}$, and the quotient group is called the \textit{Griffiths group}, denoted by
\begin{align*}
    \text{Griff}^i(X):=\CH^i(X)_{\hom}/\CH^i(X)_{\alg}.
\end{align*}

In the present article, we use properties of birational invariants related to algebraic cycles. The first and most important of these is that the group of zero-cycles of a smooth projective variety is a stably birational invariant. In particular, for a stably rational variety $X$, we have $\CH_0(X) = \mathbb{Z}$. The converse does not hold in general, but this condition can still be used to detect when a variety is not (stably) rational, through the so-called \emph{decomposition of the diagonal}. To that end, we introduce the notion of a universally trivial $\CH_0$-group.

\begin{defi}
Let $X$ be a smooth projective variety over $k$ of dimension $n$.
\begin{enumerate}
    \item We say that $X$ has a \textup{universally trivial $\CH_0$-group} if $X$ has a zero-cycle $x$ of degree 1 such that $\CH_0(X_K) = \mathbb{Z} x$ for any field extension $k \subset K$.
    
    \item We say that $X$ admits a \textup{Chow decomposition of the diagonal} if one can write
    \[
    \Delta_X = X \times x + Z \in \CH^n(X \times X),
    \]
    where $Z$ is a cycle on $X \times X$ supported on $D \times X$, with $D \subset X$ a proper closed algebraic subset.
\end{enumerate}
\end{defi}

These two conditions are equivalent (see eg. \cite[Proposition 3.11]{Voi19}). Zero-cycles also encode information about codimension 2 cycles and curves in $X$, depending on their support, as studied by Bloch and Srinivas in \cite{BS83}. To make this precise, let us consider a closed subvariety $Z \subset X$ such that for every field extension $L/k$, we have $\CH_0((X \setminus Z)_L) = 0$. Under these assumptions, by \cite[Proposition 1]{BS83}, there exist $N \in \mathbb{N}^{\geq 1}$, a divisor $D \subset X$, and cycles $\Gamma_1, \Gamma_2 \in \CH^n(X \times X)$, with $\Gamma_1$ supported on $X \times Z$ and $\Gamma_2$ supported on $D \times X$, such that\footnote{Note that if $N = 1$ and $Z = z$ is a point, we recover the classical decomposition of the diagonal.}
\[
N \cdot \Delta_X = \Gamma_1 + \Gamma_2 \in \CH^n(X \times X).
\]

Depending on the dimension of $Z$, we obtain the following consequences:

\begin{table}[h]
\centering
\begin{tabular}{|c|c|}
\hline
\textbf{dim $Z$} & \textbf{Conclusion about codimension 2 cycles on $X$} \\ \hline
$0$ & $\CH_0$-universally trivial \\ \hline
$\leq 1$ & $\CH^2(X)_{\text{alg}}$ is representable \\ \hline
$\leq 2$ & $\text{Griff}^2(X) = 0$, and there exists $N \in \mathbb{N}^{\geq 1}$ such that $N \cdot \text{Griff}_1(X) = 0$ \\ \hline
$\leq 3$ & The cycle class map $c^2: \CH^2(X) \twoheadrightarrow \text{Hdg}^4(X, \mathbb{Z})$ is surjective \\ \hline
\end{tabular}

\vspace{2mm}

\caption{Summary of results from \cite[Theorems 1 and 2]{BS83}.}
\end{table}

Let us first recall the notion of the level of a Hodge structure:

\begin{defi}
    Let $\displaystyle H_\mathbb{C}=\bigoplus_{p+q=k} H^{p,q}$ be a Hodge structure of weight $k$. The associated Hodge filtration is
    \[
        F^r(H)=\bigoplus_{p\geq r} H^{p,q}.
    \]
    The \emph{level} of a nonzero Hodge structure $H$ is the integer
    \[
        \ell := \max \{\, |p-q|,\;  H^{p,q}\neq 0\}.
    \]
    i.e., a Hodge structure $H$ of weight $k$ has level at most $k-2c$ if and only if $F^c H=H$.
\end{defi}

Recall that if $G$ is an abelian group and $X$ a smooth complex projective variety, the coniveau filtration on $\operatorname{H}^k(X,G)$ is defined by
\[
\operatorname{N}^c \operatorname{H}^k(X,G) = \sum_Z \ker \left\{\operatorname{H}^k(X,G) \to \operatorname{H}^k(X \setminus Z,G)\right\},
\]
where the sum ranges over all closed subvarieties $Z \subset X$ of codimension at least $c$ (see e.g. \cite{BO74}). With this notation, we introduce the various Jacobians that will play a role in our analysis.

\begin{defi}
    Let $H$ be a pure Hodge structure of odd weight $2m+1$ with $m \in \N$. The \emph{intermediate Jacobian} of $H$ is the complex torus
    \[
        J(H)=\overline{F^{m+1}H}/H_\Z.
    \]
    In particular, we introduce the following notations:
    \begin{enumerate}
        \item For $X \in \mathrm{SmProj}_\C$ and $H=\operatorname{H}^{2m+1}(X,\Q)$, the intermediate Jacobian of $X$ is
        \[
            J^{2m+1}(X) := J(H).
        \]
        \item For $H = \operatorname{N}^1\operatorname{H}^{2m+1}(X,\Q)$, we define the \emph{algebraic intermediate Jacobian} as
        \[
            J^{2m+1}_{\alg}(X) := J(H).
        \]
        \item Let $H^{2m+1}_{\mathrm{Hdg}}(X)$ be the maximal Hodge substructure of level $1$ contained in $\operatorname{H}^{2m+1}(X,\Q)$. Then we set $J^{2m+1}_H(X) := J\!\left(H^{2m+1}_{\mathrm{Hdg}}(X)\right)$.
        \item The \emph{Walker intermediate Jacobian} of $X$ is defined by
        \[
            J^{2m+1}_W(X) := J(\operatorname{N}^{m+1} \operatorname{H}^{2m+1}(X,\Q)).
        \]
    \end{enumerate}
\end{defi}

\subsection{\'Etale Chow groups}\label{section:Chow etale} We consider an étale version of Chow groups, namely the well-known Lichtenbaum cohomology groups, which are defined as the hypercohomology of the étale sheafification of Bloch’s cycle complex. These groups were characterized by Rosenschon and Srinivas in \cite{RS} using étale hypercoverings.

In this context, we let $\text{Sm}_k$ denote the category of smooth separated $k$-schemes over a field $k$. We denote by $z^i(X,\bullet)$ the cycle complex of abelian groups defined by Bloch:
\begin{align*}
z^i(X,\bullet): \cdots \to z^i(X,j) \to \cdots \to z^i(X,1)\to z^i(X,0) \to 0,
\end{align*}
where $z^i(X,j)$ is the subgroup generated by closed subvarieties of codimension $i$ of $X\times \Delta^j$ that intersect properly with $X \times F$ for each face $F$ of the standard simplex $\Delta^j$, and the differentials are given by the alternating sum of the pullbacks of the face maps.

We remind the reader that the homology groups of this complex define the higher Chow groups, namely $\text{CH}^i(X,2i-j) := \text{H}_j(z^i(X,\bullet))$, and in particular, $\CH^i(X) \simeq \CH^i(X,0)$.

Let us recall that $z^i(X,j)$ and the complex $z^i(X,\bullet)$ are covariantly functorial for proper morphisms and contravariantly functorial for flat morphisms between smooth $k$-schemes (see \cite[Proposition 1.3]{Blo}). Therefore, for any topology $t \in \left\{\text{fppf},\ \et,\ \text{Nis},\ \text{Zar} \right\}$, we obtain a complex of $t$-presheaves
\[
z^i(-,\bullet)\colon U \mapsto z^i(U,\bullet).
\]
In particular, for each $j$, the presheaf $z^i(-,j)\colon U \mapsto z^i(U,j)$ is a sheaf for all topologies $t \in \left\{\text{fppf},\ \et,\ \text{Nis},\ \text{Zar} \right\}$ (see \cite[Lemma 3.1]{Ge04}), and hence $z^i(-,\bullet)$ is a complex of sheaves on the small étale, Nisnevich, and Zariski sites of $X$.

We define the complex of $t$-sheaves
\[
R_X(i)_t := \left(z^i(-,\bullet)_t \otimes R \right)[-2i],
\]
where $R$ is an abelian group. In what follows, we will restrict to the cases $t = \text{Zar}$ or $t = \et$, and compute the corresponding hypercohomology groups $\mathbb{H}^j_t(X, R_X(i)_t)$.

For instance, taking $t = \text{Zar}$ and $R = \Z$, it follows from \cite[Proposition 4.2.9]{Voe00} that the hypercohomology of the complex recovers the higher Chow groups:
\[
\CH^i(X,2i-j) \simeq \mathbb{H}^j_{\text{Zar}}(X, \Z(i)).
\]
We denote the motivic and Lichtenbaum cohomology groups with coefficients in $R$ as
\[
\text{H}_M^j(X, R(i)) = \mathbb{H}^j_{\text{Zar}}(X, R(i)), \quad \text{H}_{M,\et}^j(X, R(i)) = \mathbb{H}^j_{\et}(X, R(i)),
\]
and in particular, we set $\CH^i_\et(X) := \text{H}^{2i}_{M,\et}(X, \Z(i))$.

Let $\pi \colon X_\et \to X_{\text{Zar}}$ be the canonical morphism of sites. The adjunction morphism $\Z_X(i) \to R\pi_* \pi^* \Z_X(i) = R\pi_* \Z_X(i)_\et$ induces the \emph{comparison morphisms}
\[
\kappa^{j,i} \colon \text{H}_M^j(X, \Z(i)) \longrightarrow \text{H}_{M,\et}^j(X, \Z(i))
\]
for all bidegrees $(j,i) \in \Z^2$.

Thanks to \cite[Theorem 6.18]{VV}, we know that the comparison map $\kappa^{j,i}$ is an isomorphism for $j \leq i+1$ and a monomorphism for $j \leq i+2$. For simplicity, we write 
\[
\kappa^i := \kappa^{2i,i}:\CH^i(X)\to \CH^i_\et(X)
\]
for the comparison map between Chow and Chow \'etale groups.

In some cases, it is possible to obtain more information about the Lichtenbaum cohomology groups and the comparison map between them and the higher Chow groups. For instance, there is a quasi-isomorphism $R_X(0)_\et = R$, where the right-hand side is regarded as an étale sheaf. Thus, the Lichtenbaum cohomology agrees with the usual étale cohomology, i.e.,
\[
\text{H}_{M,\et}^j(X, R(0)) \simeq \text{H}^j_\et(X, R)
\]
for all $j \in \mathbb{N}$, and in particular, $\text{CH}^0_\et(X) = \Z^{\pi_0(X)}$.

In the next case, $i = 1$, since there is a quasi-isomorphism of complexes $\Z_X(1)_\et \simeq \mathbb{G}_m[-1]$, we obtain the following isomorphisms:
\[
\text{CH}^1(X) \simeq \text{CH}_\et^1(X) = \text{Pic}(X), \quad
\text{H}_{M,\et}^3(X, \Z(1)) \simeq \text{H}_\et^3(X, \mathbb{G}_m[-1]) = \text{Br}(X),
\]
where $\text{Pic}(X)$ and $\text{Br}(X)$ denote the Picard and Grothendieck-Brauer groups of $X$, respectively. In fact, for bidegree $(j,1)$, by \cite[Corollary 3.4.3]{V}, there exists an isomorphism
\[
\text{H}_M^j(X, \Z(1)) \simeq \text{H}_{\text{Zar}}^{j-1}(X, \mathbb{G}_m),
\]
because the quasi-isomorphism $\Z_X(1) \simeq \mathbb{G}_m[-1]$ also holds in the Zariski topology.

As a particular case, consider
\[
\text{H}_M^3(X, \Z(1)) \simeq \text{H}_{\text{Zar}}^2(X, \mathbb{G}_m) = 0,
\]
since $\text{H}_M^j(X, \Z(i)) = 0$ for $j > 2i$, see \cite[Vanishing Theorem 19.3]{MVW}, whereas the Grothendieck-Brauer group $\text{Br}(X)$ is not always trivial (for instance, if $X$ is an Enriques surface, see e.g. \cite{Bea09}).

In bi-degree $(4,2)$, the comparison map is known to be injective but, in general, not surjective. More precisely, we have a short exact sequence
\[
0 \to \text{CH}^2(X) \xrightarrow{\kappa^2} \text{CH}^2_\et(X) \to \text{H}^3_{\text{nr}}(X, \Q/\Z(2)) \to 0,
\]
where $\mathcal{H}^3_\et(\Q/\Z(2))$ is the Zariski sheaf associated to the presheaf $U \mapsto \text{H}_\et^3(U, \Q/\Z(2))$, and the unramified part is given by the global sections $\text{H}_{\text{nr}}^3(X, \Q/\Z(2)) = \Gamma(X, \mathcal{H}^3_\et(\Q/\Z(2)))$. For a proof, see \cite[Proposition 2.9]{Kahn}.

\begin{remark}
For rational coefficients, the adjunction morphism
\[
\Q_X(i) \to R\pi_* \Q_X(i)_\et
\]
turns out to be an isomorphism (see \cite[Théorème 2.6]{Kahn}). Therefore, we have
\[
\text{H}_M^j(X, \Q(i)) \simeq \text{H}_{M,\et}^j(X, \Q(i))
\]
for all $(j,i) \in \Z^2$. In particular,
\[
\text{CH}^i(X)_\Q \simeq \text{CH}_\et^i(X)_\Q
\]
for all $i \in \N$.
\end{remark}

If $R$ is a torsion group, then the Lichtenbaum cohomology can be computed via étale cohomology. More precisely, let $\ell$ be a prime number, $r \in \N^{\geq 1}$, and $R = \Z/\ell^r$. Then there is a quasi-isomorphism
\[
(\Z/\ell^r)_X(i)_\et \xrightarrow{\sim} \mu_{\ell^r}^{\otimes i}.
\]

Passing to the direct limit, we also obtain a quasi-isomorphism
\[
(\Q_\ell/\Z_\ell)_X(i)_\et \xrightarrow{\sim} \varinjlim_r \mu_{\ell^r}^{\otimes i},
\]
and finally we define
\[
(\Q/\Z)_X(i)_\et := \bigoplus_\ell (\Q_\ell/\Z_\ell)_X(i)_\et \xrightarrow{\sim} \Q/\Z(i)_\et.
\]

The following result is a well-known consequence, referred to as the \emph{Suslin rigidity theorem}, concerning the morphism $\Z_X(i) \to R\pi_* \Z_X(i)_\et$ for $i \geq \dim(X)$ over $k = \bar{k}$.

\begin{prop}[{\cite[Theorem 4.2]{sus}, \cite[Section 2]{Geis}}]\label{lemGe}
Let $X$ be a smooth projective variety of dimension $n$ over an algebraically closed field $k$. Then, for $i \geq n$, the canonical morphism of sites $\pi \colon X_\et \to X_{\text{Zar}}$ induces a quasi-isomorphism of complexes of Zariski sheaves:
\[
\Z_X(i) \to R\pi_* \Z_X(i)_\et.
\]
In particular, $\operatorname{CH}^n(X) \simeq \operatorname{CH}_\et^n(X)$.
\end{prop}

Throughout the present article, we will use various cycle class maps, ranging from étale motivic and Lichtenbaum cohomology groups to Betti and étale cohomology:
\[
\begin{aligned}
&c^i\colon \CH^i(X) \to \text{H}^{2i}_B(X,\Z(i)), \quad (\text{resp. } c^i_{\Q}\colon \CH^i(X)_{\Q} \to \text{H}^{2i}_B(X,\Q(i))) \\
&c^i_{\et}\colon \CH^i_\et(X) \to \text{H}^{2i}_B(X,\Z(i)),
\end{aligned}
\]
the cycle class map and the Lichtenbaum cycle class map (see \cite{RS}). These maps satisfy the property that, as in the classical case of Chow groups, we have
\[
\operatorname{Im}(c^i_\et) \subset \operatorname{Hdg}^{2i}(X, \Z),
\quad \text{but} \quad
\operatorname{Im}(c^i_\et)_{\text{tors}} \simeq \operatorname{Hdg}^{2i}(X, \Z)_{\text{tors}}.
\]

\begin{remark}
Let us briefly discuss the functoriality of étale Chow groups: Since we work over a base field $k$ with $\operatorname{char}(k) = 0$, we can use an equivalent definition (see \cite[Theorem 7.1.2]{CD16}) of étale Chow groups in terms of the triangulated category of étale motives $\operatorname{DM}_\et(k, \Z)$. For more details, we refer the reader to \cite{Ayo} and \cite{CD16}. One defines the \textit{étale motivic cohomology} group of bidegree $(j,i)$ as
\[
\text{H}_{M,\et}^j(X, \Z(i)) := \operatorname{Hom}_{\operatorname{DM}_\et(k, \Z)}(M_\et(X), \Z(i)[j]).
\]
Thanks to the six functor formalism (see \cite{CD19}), we can recover the standard formalism for Chow groups (pull-backs, push-forwards of proper maps, intersection products, and the projection formula), now in the context of étale Chow groups. See \cite{RoSo} for further details on the construction and compatibility properties.
\end{remark}

An \emph{étale correspondence} from $X$ to $Y$ is a cycle $Z \in \CH_\et(X \times Y)$. A correspondence $Z \in \CH_\et^t(X \times Y)$ acts on cycles in $X$ via
\[
Z(T) := (\operatorname{pr}_Y)_*(Z \cdot (T \times Y)) \in \CH_\et^{i + t - n}(Y),
\]
where $T \in \CH_\et^i(X)$ and $n = \dim(X)$.

\begin{defi}
We say that a cycle $Z \in \CH_\et^i(X)$ is \emph{algebraically trivial} if there exists a smooth irreducible curve $C$, a cycle $W \in \CH_\et^i(C \times X)$, and two points $a, b \in C$ such that $W(a) = 0$ and $W(b) = Z$. We denote the subgroup of algebraically trivial cycles of codimension $i$ by $\CH_\et^i(X)_{\alg}$.
\end{defi}

We denote by $\CH_\et^i(X)_{\text{hom}} := \ker(c^i_\et)$. It is clear that
\[
\CH_\et^i(X)_{\text{alg}} \subset \CH_\et^i(X)_{\text{hom}},
\]
although it is not always obvious that this inclusion is strict. However, $\CH_\et^i(X)_{\text{alg}}$ is always a divisible group, since $\CH_\et^i(X)_{\text{hom}}$ is the maximal divisible subgroup of $\CH_\et^i(X)$ \cite[\S 3]{Gei17}. The corresponding quotient is called the \textit{étale Griffiths group}, and is denoted as
\[
\operatorname{Griff}_\et^i(X) := \CH_\et^i(X)_{\text{hom}} / \CH_\et^i(X)_{\text{alg}}.
\]

\subsection{Unramified cohomology}\label{section:unramified}
Recall that if $G$ is an abelian group and $X$ a complex algebraic variety, then $\mathcal{H}^q(G)$ denotes the Zariski sheaf associated to the presheaf $U \mapsto \operatorname{H}_B^q(U,G)$. The unramified cohomology groups are defined as
\[
\operatorname{H}_{\operatorname{nr}}^i(X,G) := \operatorname{H}_{\operatorname{Zar}}^0(X,\mathcal{H}^i(G)),
\]
which are birational invariants and relate to the torsion in Chow groups. More generally, if $\mathcal{G}$ is an étale sheaf on $X$, we set
\[
\operatorname{H}_{\operatorname{nr}}^i(X,\mathcal{G}) := \operatorname{H}^0_{\operatorname{Zar}}(X,\mathcal{H}^i(\mathcal{G})),
\]
where $\mathcal{H}^i(\mathcal{G})$ is the Zariski sheaf associated to $U \mapsto \operatorname{H}^i_{\et}(U,\mathcal{G})$.

\section{\texorpdfstring{Quartic double fivefolds}{Quartic double fivefolds}}\label{section:Other Fanos}

In this section we study algebraic cycles on the Fano fivefold given by the double cover of $\mathbb{P}^5$ branched along a smooth quartic hypersurface. Throughout, $f\colon X\to \mathbb{P}^5$ denotes such a \emph{quartic double fivefold}, branched along $B=\{q=0\}$ with $q\in \operatorname{H}^0(\mathbb{P}^5,\mathcal{O}_{\mathbb{P}^5}(4))$. Equivalently, $X$ is the smooth hypersurface
\[
X=\{y^2=q(x_0,\ldots,x_5)\}\subset \mathbb{P}(1^6,2)
\]
of degree four, which does not meet the singular point of $\mathbb{P}(1^6,2)$. In the notation of Appendix~\ref{appendixA}, $X=V(s)\subset \mathbb{P}(\mathcal{O}_{\mathbb{P}^5}\oplus\mathcal{O}_{\mathbb{P}^5}(2))$ with $s=y_0^2-y_1^2q$.

We denote by $H$ the hyperplane class of $\mathbb{P}^5$, we set $h:=f^\ast H$, and we denote by $R=\{y=0\}\subset X$ the ramification divisor, so that $f|_R\colon R\to B$ is an isomorphism and $f^\ast B=2R$. The \emph{degree} of a cycle $Z$ on $X$ is $\deg(Z)=Z\cdot h^{\dim Z}$.

One of the main features that we will exploit is the fact that such a quartic double fivefold $X$ is unirational (see\footnote{Note that this statement is not affected by the erratum \cite{Bea16a}.} eg. \cite[Proposition 5]{Bea15}): for $x\in X$ and a line $L\subset \mathbb{P}^5$ through $f(x)$, the curve $f^{-1}(L)$ has genus one, and the point $y\in f^{-1}(L)$ such that $y+f^\ast(p)\sim 3x$ for $p\in L$ defines a rational map to $X$ which is dominant when $x$ runs over the (rational) preimage of a general plane of $\mathbb{P}^5$ and $L$ over the lines through $f(x)$.

\subsection{General properties}

\begin{lemma}\label{lemma:QDF-invariants}
Let $X$ be a quartic double fivefold.
\begin{enumerate}
    \item $\operatorname{Pic}(X)=\Z h$, $-K_X=4h$ and $h^5=2$, i.e. $X$ is a del Pezzo fivefold of degree two.
    \item The cohomology $\operatorname{H}^\ast(X,\Z)$ is torsion-free, $f^\ast\colon \operatorname{H}^k(\mathbb{P}^5,\Z)\to \operatorname{H}^k(X,\Z)$ is an isomorphism for $k<5$, and $\operatorname{H}^{2k}(X,\Z)\simeq \Z$ for $0\leq k\leq 5$. More precisely, $\operatorname{H}^{2k}(X,\Z)$ is generated by $h^k$ for $k\leq 2$ and by $\frac{1}{2}h^k$ for $k\geq 3$. In particular $\operatorname{H}^6(X,\Z)=\Z\beta$ with $\beta:=\frac12 h^3$.
    \item $\operatorname{H}^5(X,\Z)$ has rank $182$, and $h^{5,0}(X)=0$, $h^{4,1}(X)=1$, $h^{3,2}(X)=90$.
\end{enumerate}
\end{lemma}

\begin{proof} Assertion (1) is classical from Fujita's classification (see e.g. \cite{Fuj90}).
The first two assertions in (2) are the Lefschetz theorem for cyclic coverings of $\mathbb{P}^5$ branched along a smooth hypersurface \cite{Cor81}. Hence, for $k\leq 2$, the group $\operatorname{H}^{2k}(X,\Z)$ is generated by $h^k$, and the remaining assertions of (2) follow from Poincaré duality. For (3), Griffiths' residue calculus for quasi-smooth hypersurfaces of weighted projective spaces \cite{Dol82} identifies $\operatorname{H}^{5-p,p}(X)$ with the degree $4p-4$ part of the Jacobian ring
\[
\C[x_0,\ldots,x_5]/(\partial q/\partial x_0,\ldots,\partial q/\partial x_5)
\]
of $q$. This gives the stated dimensions.
\end{proof}

\begin{remark}\label{rmk:CY-type}
By Lemma~\ref{lemma:QDF-invariants}(3), the Hodge structure $\operatorname{H}^5(X,\Q)$ has level three and with Hodge numbers $(1,90,90,1)$, i.e., is numerically similar to that of a Calabi-Yau threefold and it is one of the candidates for being a Fano manifold of Calabi--Yau type in the sense of Iliev and Manivel \cite{IM15} (see also \cite[Remark 1.4]{Abu20} for a discussion on the remaining conditions for being such a variety). In particular, the intermediate Jacobian $J^5(X)$ is a complex torus of dimension $91$, the Abel--Jacobi map is far from being surjective on algebraically trivial cycles, and for general $X$ the group $\operatorname{Griff}^3(X)\otimes \Q$ is infinitely generated \cite[Theorem 4.5]{FIK14}.
\end{remark}

\begin{remark}\label{rmk:lines-unirational}
An \emph{$f$-line} (resp.\ an \emph{$f$-plane}) of $X$ is a curve (resp.\ a surface) of degree one. It maps isomorphically onto a line (resp.\ a plane) $\Lambda$ of $\mathbb{P}^5$, and it is the subvariety $\{y=\sigma\}$ of $f^{-1}(\Lambda)$ for a section $\sigma\in\operatorname{H}^0(\Lambda,\mathcal{O}_\Lambda(2))$ with $q|_\Lambda=\sigma^2$, see Appendix~\ref{appendixA}, and it is contained in $R$ if and only if $\sigma=0$, i.e., if and only if $\Lambda\subset B$. The Fano scheme $F_1(X)$ of $f$-lines is nonempty, of dimension at least $\delta(5,2,1)=6$, with equality for general $X$ \cite[Proposition V.3.8]{ska} (cf.\ \cite{Tih} for quartic double solids). Alternatively, by Mori's bend-and-break \cite[Theorem V.1.6]{Kol96}, through every point of $X$ there is a rational curve $C$ with $-K_X\cdot C\leq 6$, which is an $f$-line since $-K_X=4h$. In particular, there are $f$-lines which are not contained in $R$. 

On the other hand, the Fano scheme $F_2(X)$ of $f$-planes has expected dimension $\delta(5,2,2)=0$, but it is not clear \emph{a priori} for us that it is non-empty for every such $X$. For the purpose of construction of odd degree surfaces we preferred to consider some special cubic surfaces since we expect that this construction can be extended to other situations, and since it can be used to prove that $F_2(X)$ is non-empty for every $X$ (see Remark~\ref{rmk:f-planes}).
\end{remark}

\begin{prop}\label{prop:kappa-easy}
Let $X$ be a quartic double fivefold. Then $\CH^0(X)\simeq\Z$, $\CH^1(X)=\Z h$ and $\CH^5(X)\simeq \Z$, the cycle class maps $c^i$ are isomorphisms for $i\in\{0,1,5\}$, and so are the comparison maps $\kappa^0$, $\kappa^1$ and $\kappa^5$.
\end{prop}

\begin{proof}
The assertions for $i=0,1$ follow from Lemma~\ref{lemma:QDF-invariants} and \S\ref{section:Chow etale}. Since $X$ is unirational, $\CH_0(X)\simeq \Z$, and $\kappa^5$ is an isomorphism by Proposition~\ref{lemGe}.
\end{proof}

\subsection{Cycles of codimension two and four}

We will use repeatedly the following elementary observation.

\begin{lemma}\label{lemma:six-term}
Let $\kappa\colon A\to B$ be a homomorphism of abelian groups such that $\kappa\otimes \Q$ is an isomorphism. Then there is an exact sequence
\[
0\to \ker(\kappa)\to A_{\tors}\to B_{\tors}\to \operatorname{coker}(\kappa)\to A\otimes \Q/\Z\to B\otimes\Q/\Z\to 0.
\]
Moreover, $A \otimes \Q/\Z\simeq \mathrm{Im}(\kappa)\otimes \Q/\Z$.
\end{lemma}

\begin{proof}
Apply $\operatorname{Tor}_\ast^{\Z}(-,\Q/\Z)$ to $0\to \ker(\kappa)\to A\to \mathrm{Im}(\kappa)\to 0$ and $0\to  \mathrm{Im}(\kappa)\to B\to \operatorname{coker}(\kappa)\to 0$, and use that $\ker(\kappa)$ and $\operatorname{coker}(\kappa)$ are torsion groups, together with $\operatorname{Tor}_1^\Z(G,\Q/\Z)=G_{\tors}$. The last assertion follows from the fact that $\ker(\kappa)$ is torsion.
\end{proof}

\begin{prop}\label{prop:codim24}
Let $X$ be a quartic double fivefold.
\begin{enumerate}
    \item The cycle class map $c^2\colon \CH^2(X)\to \operatorname{H}^4(X,\Z)=\Z h^2$ is an isomorphism.
    \item The cycle class map $c^4\colon \CH^4(X)\to \operatorname{H}^8(X,\Z)$ is surjective, $\operatorname{H}^8(X,\Z)$ being generated by the class of an $f$-line, and $\CH^4(X)_{\alg}$ is uniquely divisible.
    \item The comparison maps $\kappa^2$ and $\kappa^4$ are isomorphisms.
\end{enumerate}
\end{prop}

\begin{proof}
(1) Since $X$ is unirational, $\CH_0(X)$ is supported on a point. Hence, by the results of Bloch and Srinivas recalled in \S\ref{section:Chow}, homological and algebraic equivalence coincide on $\CH^2(X)$, and $\CH^2(X)_{\alg}$ is representable, hence zero since $\operatorname{H}^3(X,\Z)=0$ (cf.\ \cite[Theorem 1.9]{Murr}). As $c^2(h^2)$ generates $\operatorname{H}^4(X,\Z)$, $c^2$ is an isomorphism.

(2) An $f$-line $\ell\subset X$ satisfies $[\ell]\cdot h=1$, hence its class generates $\operatorname{H}^8(X,\Z)$ by Lemma~\ref{lemma:QDF-invariants}. The group $\CH^4(X)_{\alg}$ is divisible, and it is torsion-free: since $\CH_0(X)$ is supported on a point, the Walker Abel--Jacobi map induces an isomorphism between $(\CH^4(X)_{\alg})_{\tors}$ and the torsion of $J(\operatorname{N}^3\operatorname{H}^7(X,\Z))$ \cite[Theorem 1.1]{Suz}, and $\operatorname{H}^7(X,\Z)=0$.

(3) By \cite[Proposition 2.9]{Kahn}, $\operatorname{coker}(\kappa^2)\simeq \operatorname{H}^3_{\operatorname{nr}}(X,\Q/\Z(2))$, while $\kappa^2$ is injective (\S\ref{section:Chow etale}). Since $\CH_0(X)$ is supported on a point, $\operatorname{H}^3_{\operatorname{nr}}(X,\Q/\Z(2))$ is isomorphic to the torsion subgroup of $\operatorname{Hdg}^4(X,\Z)/\operatorname{Im}(c^2)$ \cite[Théorème 3.9]{CTV}, which vanishes by (1). For $\kappa^4$ we apply Lemma~\ref{lemma:six-term}. The group $\CH^4_\et(X)_{\tors}$ is isomorphic to $\operatorname{H}_\et^7(X,\Q/\Z(4))=0$ \cite[Proposition 5.1]{RS}, and $\CH^4_\et(X)_{\hom}$ is divisible \cite{Gei17}. Therefore, as $\CH^4_\et(X)_{\hom}$ is a $\mathbb{Q}$-vector space, we have a canonical splitting $\CH^4_\et(X)=\Z\,\kappa^4([\ell])\oplus \CH^4_\et(X)_{\hom}$ and $\CH^4_\et(X)\otimes\Q/\Z\simeq\Q/\Z$, while $\CH^4(X)\otimes\Q/\Z\simeq \Q/\Z\oplus \operatorname{Griff}^4(X)\otimes\Q/\Z$ by (2). Since $\operatorname{Griff}^4(X)$ is a torsion group by the results of Bloch and Srinivas, Lemma~\ref{lemma:six-term} gives $\operatorname{coker}(\kappa^4)=0$ and $\ker(\kappa^4)=\CH^4(X)_{\tors}$. Finally, $\CH^4(X)_{\tors}\to \operatorname{Griff}^4(X)$ is injective as $\CH^4(X)_{\alg}$ is torsion-free and $\operatorname{Griff}^4(X)=0$ by \cite[Theorem 2.3]{MP17}, which applies to $X$ as it is a smooth hypersurface of degree four in $\mathbb{P}(1^6,2)$, of dimension five and index four (cf. \cite[Theorem VI.1.6]{ska}, whose proof relies on \cite[Theorem 1.3]{TZ14}). 
\end{proof}

\subsection{Surfaces and the integral Hodge conjecture}\label{subsec:surfaces}

By Lemma~\ref{lemma:QDF-invariants}, the group $\operatorname{H}^6(X,\Z)$ is generated by $\beta=\frac{1}{2}h^3$, so the integral Hodge conjecture for $X$ reduces to the existence of an algebraic cycle of odd degree in codimension three. Note that preimages of planes of $\mathbb{P}^5$ have class $h^3=2\beta$, and if $B$ is very general then every surface $\Sigma\subset B$ has even degree by the Noether--Lefschetz theorem (see e.g.\ \cite{Voi}), so that the corresponding surface of $R\simeq B$ has class $\deg(\Sigma)\beta$. 

In particular, a surface of odd degree on $X$ cannot be contained in $R$ in general. More precisely, if $S\subset X$ is an integral surface with $S\not\subset R$ and $[S]=(2k+1)\beta$, and if we let $T:=f(S)$, then $f|_S\colon S\to T$ is birational\footnote{Otherwise $S=f^{-1}(T)$ as cycles, and the projection formula would give $S\cdot R\cdot h=[T]\cdot H\cdot f_\ast [R]=4\deg(T)$, whereas $S\cdot R\cdot h=(2k+1)\beta\cdot 2h\cdot h=2(2k+1)$}. Consequently $\deg(T)=2k+1$ and the double cover $f^{-1}(T)\to T$ splits over the generic point of $T$, i.e.\ $q|_T$ is a square in the function field of $T$: the branch divisor $B$ meets $T$ with even multiplicity everywhere. For $k=1$ this suggests to look for cubic surfaces with this property:

\begin{defi}\label{def:totally-ramified}
Let $\Sigma\subset \mathbb{P}^5$ be either a plane or a cubic surface contained in a $3$-dimensional linear subspace $W\subset \mathbb{P}^5$ (i.e.\ an effective divisor of degree three in $W\simeq\mathbb{P}^3$). We say that $\Sigma$ is \emph{totally ramified} along $B=\{q=0\}$ if
\[
q|_\Sigma=s^2\quad \text{for some } s\in \operatorname{H}^0(\Sigma,\mathcal{O}_\Sigma(2)),
\]
where for cubic surfaces we moreover require $s\neq 0$, i.e.\ $\Sigma\not\subset B$. If $s\neq 0$, then $B\cap \Sigma=2C$ as divisors on $\Sigma$, where $C=\{s=0\}$ is a curve of degree $2\deg(\Sigma)$ (the \emph{contact curve}).
\end{defi}

\begin{lemma}\label{lemma:lifting}
Let $\Sigma\subset \mathbb{P}^5$ be totally ramified along $B$, with $q|_\Sigma=s^2$, and let $\widetilde{\Sigma}\subset f^{-1}(\Sigma)$ be the closed subscheme defined by $y=s$. Then $f|_{\widetilde\Sigma}\colon \widetilde{\Sigma}\to \Sigma$ is an isomorphism, and $[\widetilde{\Sigma}]=\deg(\Sigma)\,\beta$ in $\operatorname{H}^6(X,\Z)$. If moreover $\Sigma$ is an integral cubic surface, then $f^{-1}(\Sigma)=\widetilde{\Sigma}\cup \iota(\widetilde{\Sigma})$, where $\iota$ is the covering involution, and $\widetilde{\Sigma}\cap\iota(\widetilde{\Sigma})=\widetilde{\Sigma}\cap R$ maps isomorphically onto the contact curve, and in particular $\widetilde{\Sigma}\not\subset R$.
\end{lemma}

\begin{proof} By definition $f|_{\widetilde\Sigma}$ is an isomorphism, and in particular $f_\ast[\widetilde{\Sigma}]=[\Sigma]$ as cycles, so $[\widetilde{\Sigma}]\cdot h^2=[\Sigma]\cdot H^2=\deg(\Sigma)$ and $[\widetilde{\Sigma}]=\deg(\Sigma)\beta$ by Lemma~\ref{lemma:QDF-invariants}. The last assertion follows from $y^2-s^2=(y-s)(y+s)$, since $\iota$ acts by $y\mapsto -y$ and $s\neq 0$ on the integral surface $\Sigma$.
\end{proof}

If $\Sigma$ is a plane contained in $B$, then $s=0$ and $\widetilde\Sigma$ is the copy of $\Sigma$ in $R\simeq B$. In the terminology of Remark~\ref{rmk:lines-unirational}, the surfaces $\widetilde\Sigma$ associated with planes $\Sigma$ (with $s$ possibly zero) are exactly the $f$-planes of $X$. The existence of totally ramified cubic surfaces is not a priori clear, and the following example provides smooth quartics with such surfaces.

\begin{example}\label{ex:contact-cubic}
Let $W_0=\{x_4=x_5=0\}$ and $s_0=x_0x_3-x_1x_2$. For forms $g\in \C[x_0,\ldots,x_3]_3$ and $h,a,b\in \C[x_0,\ldots,x_5]$ of degrees $1,3,3$, set
\[
q_0:=s_0^2+g\,h+x_4\,a+x_5\,b .
\]
For general $(g,h,a,b)$ the quartic $B_0=\{q_0=0\}$ is smooth, and the smooth cubic surface $T_0=\{g=0\}\subset W_0$ is totally ramified along $B_0$, with $q_0|_{T_0}=(s_0|_{T_0})^2$ and contact curve $C_0=T_0\cap\{s_0=0\}$, a smooth curve of degree six and genus four. 
\end{example}

In order to show that totally ramified cubic surfaces exist for a general quartic, we consider the following incidence varieties:

Let $\mathbb{G}(3,5)$ be the Grassmannian of $3$-dimensional linear subspaces $W\subset\mathbb{P}^5$, with tautological subbundle $\mathcal{S}$, and let $\mathcal{B}$ be the total space of $\operatorname{Sym}^3\mathcal{S}^\vee\oplus\operatorname{Sym}^2\mathcal{S}^\vee\oplus \mathcal{S}^\vee$, whose points are quadruples $(W,g,s,h)$ with $g$, $s$ and $h$ forms of degree $3$, $2$ and $1$ on $W$. We set
\[
\widehat{J}:=\left\{(q,W,g,s,h)\in \operatorname{H}^0(\mathbb{P}^5,\mathcal{O}_{\mathbb{P}^5}(4))\times \mathcal{B} \ \middle|\  q|_W=s^2+g\,h\right\}.
\]

\begin{lemma}\label{lemma:J} With the above notation, we have:
\begin{enumerate}
    \item The projection $\widehat{J}\to \mathcal{B}$ is an affine bundle of rank $91$, and thus $\widehat J$ is smooth and irreducible of dimension $133$.
    \item The torus $(\C^\ast)^2$ acts freely on the open subset $\widehat{J}^\circ\subset \widehat J$ where $q\neq 0$, $g\neq 0$ and $s\neq 0$, by $(\lambda,\mu)\cdot(q,W,g,s,h)=(\lambda^2q,W,\mu g,\lambda s,\lambda^2\mu^{-1}h)$, and the geometric quotient $J:=\widehat{J}^\circ/(\C^\ast)^2$ is a smooth irreducible variety of dimension $131$.
    \item For $(q,W,g,s,h)\in \widehat{J}^\circ$, the cubic surface $T=\{g=0\}\subset W$ is totally ramified along $\{q=0\}$, with $q|_T=(s|_T)^2$. Conversely, if $T=\{g=0\}\subset W$ is totally ramified along $\{q=0\}$, then $q|_W=s^2+gh$ for a \emph{unique} lift $s\in \operatorname{H}^0(W,\mathcal{O}_W(2))$ of the given section on $T$ and a \emph{unique} linear form $h$. If $T$ is integral, $s$ is \emph{unique up to sign}.
\end{enumerate}
\end{lemma}

\begin{proof}
(1) For fixed $(W,g,s,h)$, the condition $q|_W=s^2+gh$ is an affine condition on $q$, whose solutions form a translate of the kernel of the surjective restriction map $\operatorname{H}^0(\mathbb{P}^5,\mathcal{O}(4))\to \operatorname{H}^0(W,\mathcal{O}_W(4))$, of dimension $91$, and local trivializations of $\mathcal{S}$ give the affine bundle structure. (2) The open subset $\widehat J^\circ$ is nonempty (see Example~\ref{ex:contact-cubic}), hence irreducible of dimension $133$. The action preserves the defining equation, and it is free on $\widehat J^\circ$ with local slices obtained by normalizing a nonzero coefficient of $g$ and one of $s$, so the quotient exists \cite{MFK94}. (3) From the exact sequences $0\to\mathcal{O}_W(d-3)\xrightarrow{\cdot g}\mathcal{O}_W(d)\to \mathcal{O}_T(d)\to 0$ and the vanishing of $\operatorname{H}^1(W,\mathcal{O}_W(k))$, the restriction map $\operatorname{H}^0(W,\mathcal{O}_W(d))\to \operatorname{H}^0(T,\mathcal{O}_T(d))$ is surjective with kernel $g\cdot \operatorname{H}^0(W,\mathcal{O}_W(d-3))$. In particular, it is bijective for $d=2$, and $q|_T=s^2$ if and only if $q|_W-s^2\in g\cdot \operatorname{H}^0(W,\mathcal{O}_W(1))$. If $T$ is integral, $s^2=s'^2$ implies that $s'=\pm s$.
\end{proof}

Hence $J$ parametrizes, over the locus of integral cubic surfaces, the triples $([q],W,T)$ with $T\subset W$ totally ramified along $\{q=0\}$, and the map $(q,W,g,s,h)\mapsto [q]$ induces a projection
\[
\pi_J\colon J\to \mathbb{P}^{125}=|\mathcal{O}_{\mathbb{P}^5}(4)|.
\]
Since the pairs $(W,T)$ form a family of dimension $27$ and total ramification imposes $\dim \operatorname{H}^0(T,\mathcal{O}_T(4))-\dim \operatorname{H}^0(T,\mathcal{O}_T(2))=21$ conditions on $q$, the expected dimension of the fibres of $\pi_J$ is $6$.

\begin{theorem}\label{thm:dominance}
The projection $\pi_J\colon J\to \mathbb{P}^{125}$ is dominant, and its general fibre has dimension six. More precisely, the projection $\widehat{J}\to \operatorname{H}^0(\mathbb{P}^5,\mathcal{O}_{\mathbb{P}^5}(4))$ is smooth at the point given by Example~\ref{ex:contact-cubic} for general $(g,h,a,b)$. In particular, a general quartic double fivefold contains a six-dimensional family of smooth cubic surfaces of class $3\beta$.
\end{theorem}

\begin{proof}
Since the quotient map $\widehat J^\circ\to J$ is smooth and surjective, it suffices to show that the projection $\psi\colon \widehat{J}\to \operatorname{H}^0(\mathbb{P}^5,\mathcal{O}(4))$, $(q,W,g,s,h)\mapsto q$, is smooth at $z_0:=(q_0,W_0,g_0,s_0,\bar h_0)$, where $q_0=s_0^2+g_0h_0+x_4a_0+x_5b_0$ is as in Example~\ref{ex:contact-cubic} and where bars will denote restrictions to $W_0$, so that $q_0|_{W_0}=s_0^2+g_0\bar h_0$. Near $W_0$, the subspaces $W_{u,v}=\{x_4=u,\ x_5=v\}$, with $u,v$ linear forms in $x_0,\ldots,x_3$, give an affine chart of $\mathbb{G}(3,5)$ which identifies $\operatorname{H}^0(W_{u,v},\mathcal{O}(d))$ with $\C_d:=\C[x_0,\ldots,x_3]_d$, through the parametrization $(x_0:\cdots:x_3)\mapsto(x_0:\cdots:x_3:u:v)$ of $W_{u,v}$. In this chart $\widehat{J}=\Phi^{-1}(0)$ for
\[
\Phi(q,u,v,g,s,h)=q(x_0,\ldots,x_3,u,v)-s^2-g\,h\in \C_4 .
\]
A tangent vector to $\operatorname{H}^0(\mathbb{P}^5,\mathcal{O}(4))\times\mathcal{B}$ at $z_0$ is a tuple $(\delta q,\delta u,\delta v,\delta g,\delta s,\delta h)$ with $\delta q\in\operatorname{H}^0(\mathbb{P}^5,\mathcal{O}(4))$, $\delta u,\delta v,\delta h\in\C_1$, $\delta g\in\C_3$ and $\delta s\in\C_2$. Writing $h_0=\bar h_0+c_4x_4+c_5x_5$ with $c_4,c_5\in\C$, we have $\partial q_0/\partial x_4|_{W_0}=c_4g_0+\bar a_0$ and $\partial q_0/\partial x_5|_{W_0}=c_5g_0+\bar b_0$, so that the differential of $\Phi$ at $z_0$ is
\begin{multline*}
d\Phi(\delta q,\delta u,\delta v,\delta g,\delta s,\delta h)=\delta q|_{W_0}+(\bar a_0+c_4g_0)\,\delta u+(\bar b_0+c_5g_0)\,\delta v\\
-2s_0\,\delta s-\bar h_0\,\delta g-g_0\,\delta h,
\end{multline*}
which is surjective since so is $\delta q\mapsto \delta q|_{W_0}$, and hence $T_{z_0}\widehat{J}=\ker(d\Phi)$. From this, the differential $d\psi(\delta q,\delta u,\delta v,\delta g,\delta s,\delta h)=\delta q$ of $\psi$ is surjective on it if and only if every $\delta q|_{W_0}\in\C_4$ is a combination of the remaining terms of $d\Phi$, that is, if and only if
\begin{align}\label{eq:star}
\C_4=s_0\,\C_2+\bar h_0\,\C_3+g_0\,\C_1+\bar a_0\,\C_1+\bar b_0\,\C_1 .
\end{align}
Let $P=\{\bar h_0=0\}\subset W_0$ and, for the sake of simplicity, we denote by $\C[P]_d:=\operatorname{H}^0(P,\mathcal{O}_P(d))$. Since $\bar h_0\C_3$ is the kernel of the restriction $\C_4\to \C[P]_4$ (cf. Proof of Lemma \ref{lemma:J}), condition \eqref{eq:star} is equivalent to
\[
\C[P]_4=\bar s\,\C[P]_2+\bar g\,\C[P]_1+\bar a\,\C[P]_1+\bar b\,\C[P]_1 ,
\]
where $\bar s,\bar g,\bar a,\bar b$ are the restrictions of $s_0,g_0,a_0,b_0$ to $P$. For general $h_0$ the plane $P$ is a general plane of $W_0$, so by Bertini's theorem $\bar s$ defines a smooth conic $\bar Q\subset P$, $\bar g$ defines an irreducible cubic curve, and $\Gamma:=\bar Q\cap\{\bar g=0\}=C_0\cap P$ consists of six distinct points. Then $\bar s\,\C[P]_2+\bar g\,\C[P]_1=\operatorname{H}^0(P,\mathcal{I}_\Gamma(4))$, where both sides have dimension nine, since $\bar s$ and $\bar g$ are coprime and the six points of $\Gamma\subset \bar Q\simeq\mathbb{P}^1$ impose independent conditions on quartics. Therefore \eqref{eq:star} is equivalent to
\[
\operatorname{ev}_\Gamma(\bar a\,\C[P]_1)+\operatorname{ev}_\Gamma(\bar b\,\C[P]_1)=\C^\Gamma\simeq \C^6,
\]
where $\operatorname{ev}_\Gamma$ denotes evaluation at the points of $\Gamma$. The six points of $\Gamma$ also impose independent conditions on plane cubics, and the restriction map $\operatorname{H}^0(\mathbb{P}^5,\mathcal{O}(3))\to\C[P]_3$ is surjective. Hence we may choose $a_0$ and $b_0$ such that $\bar a$ vanishes exactly at three points $\gamma_1,\gamma_2,\gamma_3$ of $\Gamma$ and $\bar b$ vanishes exactly at the other three points $\gamma_4,\gamma_5,\gamma_6$. As $\Gamma$ lies on the smooth conic $\bar Q$, no three of its points are collinear, and thus the evaluation at $\gamma_4,\gamma_5,\gamma_6$ is an isomorphism $\C[P]_1\to\C^3$, so that $\operatorname{ev}_\Gamma(\bar a\,\C[P]_1)=0\oplus\C^3$, and similarly $\operatorname{ev}_\Gamma(\bar b\,\C[P]_1)=\C^3\oplus 0$. Thus \eqref{eq:star} holds for such a choice of $(a_0,b_0)$, and hence the open condition \eqref{eq:star} holds true on $(g_0,h_0,a_0,b_0)$. Therefore \eqref{eq:star} holds for general $(g_0,h_0,a_0,b_0)$, together with the conditions of Example~\ref{ex:contact-cubic}.

Hence $\psi$ is smooth at $z_0$, and since smooth morphisms are open, the image of $\psi$ contains a nonempty open subset, so that $\pi_J$ is dominant, and the last assertion follows from Lemma~\ref{lemma:lifting}.
\end{proof}

The dominance of $\pi_J$ can be upgraded to a surjectivity statement for the planes of $\mathbb{P}^5$ which are totally ramified along $B$, thanks to the following incidence variety. Let $\mathbb{G}(2,5)$ be the Grassmannian of planes of $\mathbb{P}^5$ and set
\[
I:=\left\{([q],\Pi)\in \mathbb{P}^{125}\times \mathbb{G}(2,5)\ \middle|\ q|_\Pi=\sigma^2 \text{ for some } \sigma\in \operatorname{H}^0(\Pi,\mathcal{O}_\Pi(2))\right\},
\]
where $\sigma=0$ is allowed, with projection $\pi_I\colon I\to \mathbb{P}^{125}$.

\begin{prop}\label{prop:planes-incidence}
The incidence $I$ is closed, of dimension $125$, and the projection $\pi_I$ is surjective. In particular, $\pi_I^{-1}([q])$ is finite for general $[q]$.
\end{prop}

\begin{proof}
Let $\mathcal{S}_2$ be the tautological subbundle on $\mathbb{G}(2,5)$. The fibrewise squaring map $\mathbb{P}(\operatorname{Sym}^2\mathcal{S}_2^\vee)\to \mathbb{P}(\operatorname{Sym}^4\mathcal{S}_2^\vee)$ is a morphism, as $\sigma^2\neq 0$ for $\sigma\neq 0$, hence it has closed image. By definition, $I$ is the locus where the restriction $q|_\Pi$ is either zero or lies in this image, hence $I$ is closed and $\pi_I$ is proper. If $(q,W,g,s,h)\in \widehat{J}^\circ$, then $q|_\Pi=(s|_\Pi)^2$ for $\Pi=\{h=0\}\subset W$ if $h\neq0$, and for every plane $\Pi\subset W$ if $h=0$, and thus we get $\pi_I(I)\supseteq \pi_J(J)$, which is dense by Theorem~\ref{thm:dominance}, and $\pi_I$ is surjective. For a fixed plane $\Pi$, the squares form a closed cone of dimension $6$ in the $15$-dimensional space $\operatorname{H}^0(\Pi,\mathcal{O}_\Pi(4))$, and the restriction map from $\operatorname{H}^0(\mathbb{P}^5,\mathcal{O}(4))$ is surjective, and then the fibres of $I\to\mathbb{G}(2,5)$ have dimension $116$, and $\dim I=125$.
\end{proof}

We now show that totally ramified cubic surfaces exist for \emph{every} smooth quartic $B$, treating separately the quartics which contain a plane and those which do not. We will need the following elementary properties of smooth quartic fourfolds, which are probably known to experts.

\begin{lemma}\label{lemma:planes-in-B}
Let $B=\{q=0\}\subset\mathbb{P}^5$ be a smooth quartic hypersurface.
\begin{enumerate}
    \item The Fano scheme $F_2(B)$ of planes contained in $B$ is finite and reduced.
    \item For every $3$-dimensional linear subspace $W\subset\mathbb{P}^5$, $q|_W$ is neither zero nor the square of a quadratic form.
\end{enumerate}
\end{lemma}

\begin{proof}
(1) Let $P=\{x_3=x_4=x_5=0\}$ be a plane contained in $B$ and write $q\equiv x_3A_3+x_4A_4+x_5A_5$ modulo the ideal $(x_3,x_4,x_5)^2$, with $a_i=A_i|_P$. Smoothness of $B$ along $P$ means that $a_3,a_4,a_5$ have no common zero on $P$, i.e.\ they form a regular sequence in $\C[x_0,x_1,x_2]$. The tangent space of $F_2(B)$ at $[P]$ is $\operatorname{H}^0(P,N_{P/B})$, the space of linear syzygies $(l_3,l_4,l_5)$ of $(a_3,a_4,a_5)$ (see \cite[Chapter 6]{3264}), which is zero since all syzygies of a regular sequence are Koszul, hence of degree three.

(2) If $q|_W=0$, write $W=\{x_4=x_5=0\}$ and $q\equiv x_4A_4+x_5A_5$ modulo $(x_4,x_5)^2$. Then, $B$ would be singular at the common zeros of the cubic forms $A_4|_W$ and $A_5|_W$. If $q|_W=s^2$ with $s\neq 0$, then $W\subset T_pB$ for every $p$ in the quadric surface $\{s=0\}\subset W\cap B$, since $dq(p)|_W=2s(p)\,ds(p)=0$. Thus the Gauss map $B\to(\mathbb{P}^5)^\vee$, $p\mapsto T_pB$, would map a surface into the line of hyperplanes containing $W$. However, the Gauss map contracts no curve, as it is given by the partial derivatives of $q$ and pulls $\mathcal{O}(1)$ back to the ample line bundle $\mathcal{O}_B(3)$.
\end{proof}

\begin{theorem}\label{thm:every-B}
For every smooth quartic hypersurface $B\subset\mathbb{P}^5$ there is a cubic surface $T\subset W\simeq\mathbb{P}^3\subset\mathbb{P}^5$ which is totally ramified along $B$. In particular, every quartic double fivefold contains a surface of class $3\beta$ mapping isomorphically onto a cubic surface of $\mathbb{P}^5$.
\end{theorem}

\begin{proof}
Assume first that $B$ contains no plane. By Proposition~\ref{prop:planes-incidence} there is a plane $\Pi$ with $q|_\Pi=\sigma^2$, and $\sigma\neq 0$ as $\Pi\not\subset B$. Choose $W\supset \Pi$ and an equation $h$ of $\Pi$ in $W$, and lift $\sigma$ to $s\in \operatorname{H}^0(W,\mathcal{O}_W(2))$. Then $q|_W-s^2$ vanishes on $\Pi$, so $q|_W=s^2+gh$ for some cubic form $g$, and $g\neq 0$ by Lemma~\ref{lemma:planes-in-B}(2). By Lemma~\ref{lemma:J}(3), $T=\{g=0\}$ is totally ramified along $B$.

Assume now that $B$ contains a plane $P_0=\{x_3=x_4=x_5=0\}$, and keep the notation of the proof of Lemma~\ref{lemma:planes-in-B}(1). The cubic forms $a_3,a_4,a_5$ are linearly independent: otherwise, since we are taking the quotient by $(x_3,x_4,x_5)^2$, after a linear change of the coordinates $x_3,x_4,x_5$ we may assume $a_5=0$, and then $a_3$ and $a_4$ have a common zero on $P_0$, contradicting the smoothness of $B$. Let $W\supset P_0$ be a $3$-dimensional linear subspace and $t$ an equation of $P_0$ in $W$, so that $x_i|_W=\lambda_it$ for $i=3,4,5$ with $(\lambda_3,\lambda_4,\lambda_5)\neq 0$. Then $q|_W=t\,c$ for a cubic form $c$ on $W$ with $c|_{P_0}=\sum_i\lambda_ia_i\neq0$, and in particular $t\nmid c$ since the $a_i$ are linearly independent. Finally, for every nonzero linear form $u$ on $W$ we write
\[
q|_W=(tu)^2+(c-tu^2)\,t .
\]
Since $g:=c-tu^2\neq 0$, Lemma~\ref{lemma:J}(3) shows that the cubic surface $T_u=\{c-tu^2=0\}$ is totally ramified along $B$, with $q|_{T_u}=(tu|_{T_u})^2$. The last assertion follows from Lemma~\ref{lemma:lifting}.
\end{proof}

\begin{theorem}\label{thm:Hodge}
Let $f\colon X\to \mathbb{P}^5$ be a smooth double cover branched along a smooth quartic hypersurface $B\in |\mathcal{O}_{\mathbb{P}^5}(4)|$. Then the integral Hodge conjecture holds for $X$.
\end{theorem}

\begin{proof}
By the above discussion (cf. Remark~\ref{rmk:lines-unirational}), the result follows from Lemma~\ref{lemma:QDF-invariants} and Theorem~\ref{thm:every-B}.
\end{proof}

\begin{remark}\label{rmk:f-planes}
Proposition~\ref{prop:planes-incidence} and Lemma~\ref{lemma:lifting} also show that every quartic double fivefold contains $f$-planes. More precisely, for $([q],\Pi)\in I$, the surface $\widetilde{\Pi}=\{y=\sigma\}$ maps isomorphically onto $\Pi$, has class $\beta$, and is contained in $R$ if and only if $\Pi\subset B$. In particular, $F_2(X)$ is nonempty for every $X$, and it is finite for general $X$, since so is the general fibre of $\pi_I$ and each plane $\Pi$ carries at most two $f$-planes $\{y=\pm\sigma\}$. This gives another proof of Theorem~\ref{thm:Hodge}, which also relies on Theorem~\ref{thm:dominance}, through Proposition~\ref{prop:planes-incidence}. Note that the latter is also used in the proof of Theorem~\ref{thm:every-B}, when $B$ contains no plane. We have preferred to present the integral Hodge conjecture through totally ramified cubic surfaces, because the method of Theorem~\ref{thm:dominance} provides a general way to construct algebraic surfaces on double covers.  Totally ramified subvarieties of higher odd degree, as well as other cyclic coverings, will be considered in a forthcoming work.
\end{remark}

\subsection{The comparison maps and a Sarkisov link}

By Propositions~\ref{prop:kappa-easy} and \ref{prop:codim24}, the comparison maps $\kappa^i$ are isomorphisms for $i\neq 3$. We now study $\kappa^3$, by means of the Bloch-Ogus spectral sequence
\begin{align}  \label{eq:Bloch-Ogus}
E_2^{p,q}=\operatorname{H}^p(X,\mathcal{H}^q(A))\Longrightarrow \operatorname{H}^{p+q}(X,A),
\end{align}
whose abutment filtration is the coniveau filtration $\operatorname{N}^\bullet$ \cite{BO74}, with $E_2^{p,q}=0$ for $p>q$ and $\operatorname{H}^p(X,\mathcal{H}^p(\Z))\simeq \CH^p(X)/\alg$. Recall that the sheaves $\mathcal{H}^q(\Z)$ are torsion-free \cite[Théorème 3.1]{CTV}, so that for all $p,q$ there are exact sequences \cite[\S 3]{CTV}
\begin{multline}\label{eq:BK}
0\to \operatorname{H}^p(X,\mathcal{H}^q(\Z))\otimes\Q/\Z\to \operatorname{H}^p(X,\mathcal{H}^q(\Q/\Z))\\
\to \operatorname{H}^{p+1}(X,\mathcal{H}^q(\Z))_{\tors}\to 0
\end{multline}
(here and below, the twists of $\Q/\Z$ are omitted when irrelevant).
Moreover, since $\CH_0(X)$ is supported on a point, there is an integer $N\geq 1$ such that $N\cdot \operatorname{H}^q_{\operatorname{nr}}(X,A)=0$ for all $q\geq 1$ and all abelian groups $A$ \cite[Proposition 3.3]{CTV}. In particular $\operatorname{H}^q_{\operatorname{nr}}(X,\Z)=\operatorname{H}^0(X,\mathcal{H}^q(\Z))$ vanishes for $q\geq 1$, being torsion-free.

\begin{lemma}\label{lemma:kappa3}
Let $X$ be a quartic double fivefold.
\begin{enumerate}
    \item $\operatorname{H}^4_{\operatorname{nr}}(X,\Q/\Z(3))\simeq \operatorname{H}^1(X,\mathcal{H}^4(\Z))_{\tors}$, and $\ker(\kappa^3)$ is a quotient of $\operatorname{H}^4_{\operatorname{nr}}(X,\Q/\Z(3))$.
    \item $\operatorname{coker}(\kappa^3)\simeq \operatorname{H}^1(X,\mathcal{H}^4(\Z))\otimes\Q/\Z$.
    \item There is an exact sequence
    \[
    0\to \operatorname{H}^5(X,\Z)/\operatorname{N}^2\operatorname{H}^5(X,\Z)\to \operatorname{H}^1(X,\mathcal{H}^4(\Z))\to \operatorname{Griff}^3(X)\to 0 .
    \]
\end{enumerate}
\end{lemma}

\begin{proof}
By the Beilinson-Lichtenbaum conjecture \cite[Theorem 6.18]{VV}, the cone of $\Z_X(3)\to R\pi_\ast\Z_X(3)_\et$ is $\tau_{\geq 5}R\pi_\ast\Z_X(3)_\et$, and since $R\pi_\ast\Q_X(3)_\et\simeq \Q_X(3)$, we have $R^j\pi_\ast\Z_X(3)_\et\simeq \mathcal{H}^{j-1}(\Q/\Z(3))$ for $j\geq 5$ by \cite{Kahn}. As we recalled in \S~\ref{section:Chow etale}, we have $\operatorname{H}^7_M(X,\Z(3))=0$ and then we get an exact sequence
\begin{multline*}
\operatorname{H}^5_{M,\et}(X,\Z(3))\to \mathbb{H}^5(X,\tau_{\geq 5}R\pi_\ast\Z_X(3)_\et)\to \CH^3(X)\\
\xrightarrow{\kappa^3}\CH^3_\et(X)\to \mathbb{H}^6(X,\tau_{\geq 5}R\pi_\ast\Z_X(3)_\et)\to 0,
\end{multline*}
and the hypercohomology spectral sequence of the complex $\tau_{\geq 5}R\pi_\ast\Z_X(3)_\et$ has $E_2$-terms
\[
E_2^{p,j}=\operatorname{H}^p(X,\mathcal{H}^{j-1}(\Q/\Z(3)))\ \text{ for } j\geq 5,\qquad E_2^{p,j}=0\ \text{ for } j\leq 4.
\]

(1) In total degree five only $E_2^{0,5}=\operatorname{H}^4_{\operatorname{nr}}(X,\Q/\Z(3))$ survives, which proves the second assertion. The first one follows from \eqref{eq:BK} with $(p,q)=(0,4)$, since $\operatorname{H}^4_{\operatorname{nr}}(X,\Z)=0$.

(2) In total degree six we obtain an exact sequence
\[
0\to \operatorname{H}^1(X,\mathcal{H}^4(\Q/\Z))\to \operatorname{coker}(\kappa^3)\to K\to 0, \quad K:=E_\infty^{0,6}\subset \operatorname{H}^5_{\operatorname{nr}}(X,\Q/\Z(3)).
\]
By Theorem~\ref{thm:Hodge}, the cycle class map $c^3$ is surjective, i.e.\ $\operatorname{N}^3\operatorname{H}^6(X,\Z)=\operatorname{H}^6(X,\Z)$, so the terms $E_\infty^{p,6-p}$ with $p<3$ of the Bloch-Ogus spectral sequence \eqref{eq:Bloch-Ogus} with $\Z$ coefficients vanish. As $E_2^{0,5}=\operatorname{H}^5_{\operatorname{nr}}(X,\Z)=0$, no differential hits or leaves $E_2^{2,4}$, whence $\operatorname{H}^2(X,\mathcal{H}^4(\Z))=E_\infty^{2,4}=0$, and \eqref{eq:BK} gives $\operatorname{H}^1(X,\mathcal{H}^4(\Q/\Z))\simeq \operatorname{H}^1(X,\mathcal{H}^4(\Z))\otimes\Q/\Z$. On the other hand, since $c^3=c^3_\et\circ\kappa^3$ is surjective, $\operatorname{coker}(\kappa^3)$ is a quotient of $\CH^3_\et(X)_{\hom}$, which is divisible because $\operatorname{H}^6(X,\Z)$ is torsion-free (see the proof of Proposition~\ref{prop:codim24}). Hence $K$ is divisible and $N$-torsion, i.e.\ $K=0$.

(3) Since $\operatorname{H}^5_{\operatorname{nr}}(X,\Z)=0$, we have $\operatorname{N}^1\operatorname{H}^5(X,\Z)=\operatorname{H}^5(X,\Z)$, and the Bloch--Ogus spectral sequence \eqref{eq:Bloch-Ogus} gives $\ker\big(d_2\colon \operatorname{H}^1(X,\mathcal{H}^4(\Z))\to \operatorname{H}^3(X,\mathcal{H}^3(\Z))\big)=E_\infty^{1,4}=\operatorname{H}^5(X,\Z)/\operatorname{N}^2\operatorname{H}^5(X,\Z)$, while the image of $d_2$ is the kernel of the cycle class map $\CH^3(X)/\alg\to \operatorname{H}^6(X,\Z)$, that is, $\operatorname{Griff}^3(X)$.
\end{proof}

We now relate the group $\operatorname{H}^4_{\operatorname{nr}}(X,\Q/\Z(3))$ to the birational geometry of $X$. Let $\ell\subset X$ be an $f$-line which is not contained in $R$ (Remark~\ref{rmk:lines-unirational}) and let $\ell_0=f(\ell)$. Choose coordinates such that $\ell_0=\{x_2=\cdots=x_5=0\}$, then $q|_{\ell_0}=g(x_0,x_1)^2$ for a nonzero quadratic form $g$, and $\ell=\{x_2=\cdots=x_5=0,\ y=g(x_0,x_1)\}$ (cf. \cite[Observation V.3.1]{ska}). Let $\eta\colon X\dashrightarrow\mathbb{P}^3$, $(x,y)\mapsto (x_2:x_3:x_4:x_5)$, be the projection from $\ell$, let $K=\C(\mathbb{P}^3)=\C(u_3,u_4,u_5)$ with $u_i=x_i/x_2$, and define the cubic form $c\in K[x_0,x_1,t]_3$ by
\[
q(x_0,x_1,t,tu_3,tu_4,tu_5)=g(x_0,x_1)^2+t\,c(x_0,x_1,t).
\]

\begin{prop}\label{prop:sarkisov}
With the above notation:
\begin{enumerate}
    \item The generic fibre of $\eta$ is the del Pezzo surface of degree two
    \[
    X_K=\{w^2=g(x_0,x_1)^2+t\,c(x_0,x_1,t)\}\subset\mathbb{P}_K(1,1,1,2),
    \]
    which is smooth over $K$ and contains the two $K$-rational $(-1)$-curves $E_\pm=\{t=0,\ w=\pm g\}$, with $E_+\cdot E_-=2$.
    \item The contraction of $E_-$ is the morphism $X_K\to \mathbb{P}^3_K$ given by $v=(w-g)/t$, which identifies $X_K$ with the blow-up of the smooth cubic surface
    \[
    Y_K=\{t\,v^2+2g(x_0,x_1)\,v-c(x_0,x_1,t)=0\}\subset \mathbb{P}^3_K
    \]
    at its $K$-point $(0:0:0:1)$. In particular, $\C(X)\simeq K(Y_K)$.
    \item For all $i$ and $j$ there is an injection
    \[
    \operatorname{H}^i_{\operatorname{nr}}(X,\Q/\Z(j))\hookrightarrow \operatorname{H}^i_{\operatorname{nr}}(K(Y_K)/K,\Q/\Z(j)).
    \]
    In particular, if the group $\operatorname{H}^4_{\operatorname{nr}}(K(Y_K)/K,\Q/\Z(3))$ vanishes, then so does $\operatorname{H}^4_{\operatorname{nr}}(X,\Q/\Z(3))$.
\end{enumerate}
\end{prop}

\begin{proof}
(1) The closure of the fibre of $\eta$ over $(1:u_3:u_4:u_5)$ is the preimage of the plane $\Pi_u$ spanned by $\ell_0$ and $(0:0:1:u_3:u_4:u_5)$, parametrized by $(x_0:x_1:t)\mapsto (x_0:x_1:t:tu_3:tu_4:tu_5)$, which gives the equation of $X_K$, a double cover of $\mathbb{P}^2_K$ branched along the quartic curve $\Delta=\{g^2+tc=0\}$. Hence $X_K$ is smooth, and since $-K_{X_K}$ is the pullback of $\mathcal{O}_{\mathbb{P}^2}(1)$, it is a del Pezzo surface of degree two. The line $\{t=0\}$ is a bitangent of $\Delta$, as $(g^2+tc)|_{t=0}=g^2$, so its preimage splits as $E_+\cup E_-$, where $E_\pm$ are smooth rational curves with $-K_{X_K}\cdot E_\pm=1$, hence $(-1)$-curves. Since $E_++E_-\sim -K_{X_K}$ and $K_{X_K}^2=2$, we get $E_+\cdot E_-=2$.

(2) By the relation $(w-g)(w+g)=tc$, the rational map $(x_0:x_1:t:w)\mapsto(x_0:x_1:t:(w-g)/t)$ is given both by $(tx_0:tx_1:t^2:w-g)$ and by $(x_0(w+g):x_1(w+g):t(w+g):c)$. These two expressions have no common zero on $X_K$: otherwise $t=0$ and $w=g=-g$, so that $t=g=w=0$ and $c=0$, which is impossible as $c$ does not vanish at the points where $t=g=0$. Hence it is a morphism. It takes values in $Y_K$, since $w=g+tv$ gives $t v^2+2gv=c$, it contracts $E_-$ to $p=(0:0:0:1)$, and it is an isomorphism from $X_K\setminus E_-$ onto $Y_K\setminus\{p\}$, with inverse $(x_0:x_1:t:v)\mapsto (x_0:x_1:t:g+tv)$. Moreover $Y_K$ is smooth at $p$, where its tangent plane is $\{t=0\}$. Hence $Y_K$ is a smooth cubic surface and $X_K\to Y_K$ is the blow-up of $p$. This is the elementary Sarkisov link between del Pezzo surfaces of degrees two and three over non-closed fields \cite{Isk96}.

(3) Unramified cohomology is a birational invariant, $\operatorname{H}^i_{\operatorname{nr}}(X,\Q/\Z(j))=\operatorname{H}^i_{\operatorname{nr}}(\C(X)/\C,\Q/\Z(j))$ (\S\ref{section:unramified}), and every discrete valuation of $\C(X)=K(Y_K)$ which is trivial on $K$ is trivial on $\C$.
\end{proof}

We can now summarize the behaviour of the comparison maps of quartic double fivefolds, distinguishing the case of codimension three.

\begin{theorem}\label{thm:kappa}
Let $X$ be a quartic double fivefold.
\begin{enumerate}
    \item The comparison maps $\kappa^i\colon \CH^i(X)\to\CH^i_\et(X)$ are isomorphisms for $i\neq 3$.
    \item $\operatorname{coker}(\kappa^3)\simeq \operatorname{H}^1(X,\mathcal{H}^4(\Z))\otimes \Q/\Z$, and $\ker(\kappa^3)$ is a quotient of $\operatorname{H}^4_{\operatorname{nr}}(X,\Q/\Z(3))\simeq \operatorname{H}^1(X,\mathcal{H}^4(\Z))_{\tors}$. In particular, $\kappa^3$ is surjective if and only if $\operatorname{H}^1(X,\mathcal{H}^4(\Z))$ is divisible, and it is an isomorphism if $\operatorname{H}^1(X,\mathcal{H}^4(\Z))$ is uniquely divisible.
    \item Let $Y_K$ be the cubic surface over $K=\C(\mathbb{P}^3)$ associated with an $f$-line of $X$ as in Proposition~\ref{prop:sarkisov}. If $\operatorname{H}^4_{\operatorname{nr}}(K(Y_K)/K,\Q/\Z(3))=0$, then $\kappa^3$ is injective, and $\kappa^3$ is an isomorphism if and only if $\operatorname{H}^1(X,\mathcal{H}^4(\Z))$ is a $\Q$-vector space.
\end{enumerate}
\end{theorem}

\begin{proof}
(1) is Propositions~\ref{prop:kappa-easy} and \ref{prop:codim24}, and (2) is Lemma~\ref{lemma:kappa3}. Under the hypothesis of (3), $\operatorname{H}^4_{\operatorname{nr}}(X,\Q/\Z(3))=0$ by Proposition~\ref{prop:sarkisov}, so that $\kappa^3$ is injective and $\operatorname{H}^1(X,\mathcal{H}^4(\Z))$ is torsion-free by (2), and the latter is then divisible if and only if it is uniquely divisible.
\end{proof}

\begin{remark}\label{rmk:final}
We conclude with some comments on the remaining case $i=3$.
\begin{enumerate}
    \item By \cite[Corollary 1.4]{AS24}, $\operatorname{coker}(\kappa^i)$ is isomorphic to the refined unramified cohomology group $\operatorname{H}^{2i-1}_{i-2,\operatorname{nr}}(X,\Q/\Z(i))$ of \cite{Sch23}. Hence Theorem~\ref{thm:kappa} shows that these groups vanish for $i\neq 3$, and that $\operatorname{H}^5_{1,\operatorname{nr}}(X,\Q/\Z(3))\simeq \operatorname{H}^1(X,\mathcal{H}^4(\Z))\otimes\Q/\Z$.
    \item Let $\Lambda_X:=\operatorname{H}^5(X,\Z)/\operatorname{N}^2\operatorname{H}^5(X,\Z)$. Applying $\operatorname{Tor}^\Z_\ast(-,\Q/\Z)$ to the exact sequence of Lemma~\ref{lemma:kappa3}(3) and using Lemma~\ref{lemma:kappa3}, we obtain an exact sequence (cf.\ \cite{Ma22})
    \begin{multline*}
    0\to (\Lambda_X)_{\tors}\to \operatorname{H}^4_{\operatorname{nr}}(X,\Q/\Z(3))\to \operatorname{Griff}^3(X)_{\tors}\\
    \to \Lambda_X\otimes\Q/\Z\to \operatorname{coker}(\kappa^3)\to \operatorname{Griff}^3(X)\otimes\Q/\Z\to 0.
    \end{multline*}
    Since $\operatorname{N}^2\operatorname{H}^5(X,\C)\subset \operatorname{H}^{3,2}(X)\oplus\operatorname{H}^{2,3}(X)$ and $h^{4,1}(X)=1$, the group $\Lambda_X$ has rank at least two. Consequently, if $\kappa^3$ is surjective, then $\operatorname{Griff}^3(X)$ is divisible and its torsion subgroup surjects onto a group containing $(\Q/\Z)^2$. By contrast, the étale Griffiths group $\operatorname{Griff}^3_\et(X)$ is always divisible, being a quotient of the divisible group $\CH^3_\et(X)_{\hom}$.
    \item On the side of the kernel, Proposition~\ref{prop:sarkisov} reduces the vanishing of $\ker(\kappa^3)$ to the vanishing of the unramified cohomology group $\operatorname{H}^4_{\operatorname{nr}}(K(Y_K)/K,\Q/\Z(3))$ of the smooth cubic surface $Y_K$ over the field $K=\C(\mathbb{P}^3)$ of cohomological dimension three. Since $Y_K$ has a $K$-point, it is unirational over $K$ (see \cite[Chapter 1]{KSC04}). Note that if $Y_K$ had a universally trivial $\CH_0$-group over $K$, for instance if $Y_K$ were stably rational over $K$, then $\operatorname{H}^4_{\operatorname{nr}}(K(Y_K)/K,\Q/\Z(3))=\operatorname{H}^4(K,\Q/\Z(3))=0$ by \cite[Theorem 2.11]{Mer08} and \cite[Chapter II, \S 4.2]{Ser02}, and $\kappa^3$ would be injective.  On the other hand, $Y_K$ is not rational over $K$. Indeed, $X_K\setminus(E_+\cup E_-)$ is the generic fibre of $\eta$ on $X\setminus f^{-1}(\ell_0)$, whose Picard group is $\Z h$ as $f^{-1}(\ell_0)$ has codimension four, so that $\operatorname{Pic}(X_K\setminus(E_+\cup E_-))$ is generated by the restriction of $h$, which is trivial since $-K_X|_{X_K}\sim 4h|_{X_K}\sim E_++E_-$. Hence $\operatorname{Pic}(X_K)$ is generated by $E_+$ and $E_-$, the cubic surface $Y_K$ has Picard rank one, and it is not $K$-rational by Segre's theorem \cite[Chapter 2]{KSC04}. We do not claim anything about whether a smooth quartic double fivefold can be stably rational, or have a universally trivial $\CH_0$-group, and we do not know whether $\operatorname{H}^4_{\operatorname{nr}}(K(Y_K)/K,\Q/\Z(3))$ vanishes, a property which is a priori weaker than the universal triviality of $\CH_0(Y_K)$.
\end{enumerate}
\end{remark}

\appendix

\section{\texorpdfstring{Fano schemes of $f$-planes in double covers of $\mathbb{P}^n$}{Fano schemes of f-planes in double covers of Pn}} \label{appendixA}

In this appendix, we present some elementary results concerning Fano schemes for double covers of a projective space of dimension $n$ over an algebraically closed field $k$. The arguments are classical: the case of lines was treated by Tihomirov in \cite{Tih} for $n=3$ and $d=2$, and later in full generality in the Ph.D. thesis of Skauli \cite{ska}. For completeness and self-containment, we extend these results to the case of $f$-planes of dimension $r$, defined below. We also refer to \cite{Cha26} for a detailed treatment of the representability of the moduli functor in the case $n=2$, $d=2$ and $r=1$, with the notation as below.

Let $X$ be an $n$-dimensional variety, and suppose that
\[
f: X \xrightarrow{2:1} \mathbb{P}^n
\]
is a double cover branched along a hypersurface $B \subset \mathbb{P}^n$ of degree $2d$. Then
\[
K_X = f^*K_{\mathbb{P}^n} + R, \qquad f^*B = 2R,
\]
where $R$ denotes the ramification locus of $f$. Such a variety $X$ admits the following models:
\begin{enumerate}
    \item As a hypersurface of degree $2d$ in $\mathbb{P}(1^{n+1},2)$.
    \item By means of the projective bundle $p:P := \mathbb{P}(\mathcal{O}_{\mathbb{P}^n} \oplus \mathcal{O}_{\mathbb{P}^n}(d)) \to \mathbb{P}^n$, defined by a section $s$ of $\mathcal{O}_P(2)$. Let $y_1$ be a generator of $\text{H}^0(\mathcal{O}_P(1)\otimes p^*\mathcal{O}_{\mathbb{P}^n}(-d))$ and $y_0 \in \text{H}^0(\mathcal{O}_P(1))$ an element not in the image of
    \[
    \text{H}^0(\mathcal{O}_P(1)\otimes p^*\mathcal{O}_{\mathbb{P}^n}(-d)) \otimes \text{H}^0(p^*\mathcal{O}_{\mathbb{P}^n}(d)) \longrightarrow \text{H}^0(\mathcal{O}_P(1)),
    \]
    then the double cover is obtained as $f:=p|_X:X\to \mathbb{P}^n$ and the section defining $X$ is of the form $s=y_0^2-y_1^2 q(x_0,\ldots,x_n)$.
\end{enumerate}

The study of Fano scheme of lines of these varieties extends naturally to the Fano scheme of $f$-planes of dimension $r$ in $X$. We will closely follow the presentation in \cite[Chapter~V]{ska} and \cite{DM98}.

\begin{defi}\label{defi:lines}
Let $X$ be a smooth double cover of $\mathbb{P}^n$. An \emph{$f$-plane of dimension $r$} (an \emph{$f$-line} if $r=1$, an \emph{$f$-plane} if $r=2$) in $X$ is an $r$-dimensional subvariety $h_r \subset X$ such that
\[
h_r \cdot f^*H^{r} = 1,
\]
where $H^r$ is the class of a codimension-$r$ hyperplane section of $\mathbb{P}^n$.
\end{defi}

By the push-pull formula, $f_*(h_r \cdot f^*H^r) = f_*h_r \cdot H^r = 1$. Hence, the image of $h_r$ in $\mathbb{P}^n$ is an $r$-plane in the usual sense. Note that $h_r$ may be contained in the ramification divisor $R$: if the image $\Lambda$ of $h_r$ is contained in $B$, then $f^{-1}(\Lambda)_{\mathrm{red}}\subset R\simeq B$ maps isomorphically onto $\Lambda$, and it is an $f$-plane of dimension $r$ of $X$.

\begin{prop}
If $h_r \subset X \subset P$ is an $f$-plane of dimension $r$, then $h_r$ is defined in $P$ as the vanishing locus of a section of
\[
\mathcal{O}_P(1) \oplus p^*\mathcal{O}_{\mathbb{P}^n}(1)^{\oplus (n-r)}.
\]
\end{prop}

\begin{proof}
Let $h_r' \subset \mathbb{P}^n$ be the image of $h_r$. Let $s' \in \text{H}^0(\mathbb{P}^n,\mathcal{O}_{\mathbb{P}^n}(1)^{\oplus (n-r)})$ be a section whose zero locus defines $h_r'$. Then $p^*s' \in \text{H}^0(P,p^*\mathcal{O}_{\mathbb{P}^n}(1)^{\oplus (n-r)}))$ defines the ruled variety $Z := p^{-1}(h_r')$. The subvariety $h_r$ is contained in $Z$ and defined by a section of $\mathcal{O}_Z(1)$, which extends to $\mathcal{O}_P(1)$ by pushforward under $i:Z \hookrightarrow P$.
\end{proof}

In other words, in terms of the model (2) above, $h_r=\{y_0=\sigma\, y_1\}\cap p^{-1}(\Lambda)$, where $\Lambda=f(h_r)$ and $\sigma\in \operatorname{H}^0(\Lambda,\mathcal{O}_\Lambda(d))$ satisfies $q|_\Lambda=\sigma^2$. In particular, $h_r\subset R$ if and only if $\sigma=0$, i.e.\ if and only if $\Lambda\subset B$.

Let $G = \mathbb{G}(r,n)$ be the Grassmannian of $r$-planes in $\mathbb{P}^n$ and $Q$ its tautological quotient bundle of rank $n-r$. The incidence correspondence
\[
\Phi := \{(h,p) \in \mathbb{G}(r,n) \times \mathbb{P}^n \mid p \in h\}
\]
is isomorphic to $\mathbb{P}(Q^\vee)$. Denote the projections by $\psi:\Phi \to \mathbb{G}(r,n)$ and $\rho:\Phi \to \mathbb{P}^n$.

\begin{prop}
The Grassmannian of $r$-planes in $P$ is
\[
G_P(r) \simeq \mathbb{P}(\mathcal{O}_{\mathbb{G}(r,n)} \oplus \mathrm{Sym}^d Q^\vee),
\]
of dimension $(r+1)(n-r) + \binom{d+r}{d}$.
\end{prop}

\begin{proof}
Consider the Cartesian square
\[
\begin{tikzcd}
\Phi' := \Phi \times_{\mathbb{P}^n} P \arrow{d}{u} \arrow{r}{\rho'} & P \arrow{d}{p} \\
\Phi \arrow{r}{\rho} & \mathbb{P}^n
\end{tikzcd}
\]
and set $t = \psi \circ u:\Phi' \to \mathbb{G}(r,n)$. A fiber of $\gamma:G_P(r) \to \mathbb{G}(r,n)$ is the restriction of $|\mathcal{O}_P(1)|$ to $Z_h := p^{-1}(h)$, i.e., the set of $r$-planes in $t^{-1}(h)$. Thus
\[
\gamma^{-1}(h) \simeq \mathbb{P}\left(\left(t|_{Z_h}\right)_*\mathcal{O}_P(1)\right) \simeq \mathbb{P}\left(t_* \rho'^* \mathcal{O}_P(1) \big|_h \right).
\]
By base change and the formula for projective bundles,
\[
G_P(r) \simeq \mathbb{P}\left(\psi_* \rho^* p_* \mathcal{O}_P(1)\right) \simeq \mathbb{P}(\mathcal{O}_{\mathbb{G}(r,n)} \oplus \mathrm{Sym}^d Q^\vee).
\]
The dimension formula follows since this is the projectivization of a vector bundle of rank $\binom{d+r}{d} + 1$ over $\mathbb{G}(r,n)$.
\end{proof}

\begin{defi}
Let $X$ be a double cover of $\mathbb{P}^n$ ramified along a hypersurface of degree $2d$. We define the \emph{Fano scheme of $f$-planes of dimension $r$} of $X$ as
\[
F_r(X) := \{\, h_r \in G_P(r) \mid h_r \subset X \,\} \subset G_P(r).
\]
\end{defi}

At a smooth point $h_r \in F_r(X)$, the Zariski tangent space is $H^0(N_{h_r/X})$, where $N_{h_r/X}$ denotes the normal bundle. From the short exact sequence
\[
0 \to N_{h_r/X} \to N_{h_r/P}|_{h_r} \to N_{X/P}|_{h_r} \to 0,
\]
and using that $X$ and $h_r$ are cut out respectively by sections of $\mathcal{O}_P(2)$ and $\mathcal{O}_P(1) \oplus p^*\mathcal{O}_{\mathbb{P}^n}(1)^{\oplus (n-r)}$, we obtain
\[
N_{X/P} = \mathcal{O}_P(2)|_X, \qquad N_{h_r/P} = \left(\mathcal{O}_P(1) \oplus p^*\mathcal{O}_{\mathbb{P}^n}(1)^{\oplus (n-r)}\right)\big|_{h_r}.
\]
Thus,
\[
0 \to N_{h_r/X} \to \mathcal{O}_{h_r}(d) \oplus \mathcal{O}_{h_r}(1)^{\oplus (n-r)} \xrightarrow{\alpha} \mathcal{O}_{h_r}(2d) \to 0.
\]

Let $\mathfrak{X} \subset \mathbb{P}(H^0(\mathcal{O}_P(2)))$ be the parameter space of double covers of $\mathbb{P}^n$. Consider the incidence correspondence
\[
I_r := \{\, (h_r, X) \in G_P(r) \times \mathfrak{X} \mid h_r \subset X \,\}
\]
with projections $p_G: I_r \to G_P(r)$ and $p_{\mathfrak{X}}: I_r \to \mathfrak{X}$. The fiber $p_{\mathfrak{X}}^{-1}(X)$ is $F_r(X)$, while $p_G^{-1}(h_r)$ consists of double covers containing $h_r$.  

Since $\mathfrak{X} \simeq \mathrm{Sym}^{2d}(\mathcal{O}_{\mathbb{P}^n} \oplus \mathcal{O}_{\mathbb{P}^n}(d))$, one finds that $p_G^{-1}(h_r)$ is isomorphic to
\[
\ker\left\{ \mathrm{Sym}^{2d}\!\left((\mathcal{O}_{\mathbb{P}^n} \oplus \mathcal{O}_{\mathbb{P}^n}(d))^\vee\right) \to \mathrm{Sym}^{2d}\!\left((\mathcal{O}_P(1) \oplus p^*\mathcal{O}_{\mathbb{P}^n}(1)^{\oplus (n-r)})^\vee\right) \right\}.
\]
This kernel has codimension $\binom{2d+r}{2d}$, hence $I_r$ has codimension $\binom{2d+r}{2d}$ in $G_P(r) \times \mathfrak{X}$. It follows that the \emph{expected dimension} of $F_r(X)$ is
\[
\delta(n,d,r) := (r+1)(n-r) + \binom{d+r}{d} - \binom{2d+r}{2d}.
\]

\begin{prop}
Let $X$ be a double cover of $\mathbb{P}^n$ branched along a hypersurface of degree $2d$.
\begin{enumerate}
    \item If $\delta(n,d,r) < 0$, then for a general $X$ the scheme $F_r(X)$ is empty.
    \item If $\delta(n,d,r) \geq 0$, then $F_r(X)$ is non-empty.
\end{enumerate}
\end{prop}

\begin{proof}
If $\delta(n,d,r) < 0$, then $\dim I_r < \dim \mathfrak{X}$, so for a general $X$ the fiber $p_{\mathfrak{X}}^{-1}(X)$ is empty. The second statement follows from the standard argument consisting in showing the existence of an explicit $(h_r,X)\in I_r$ such that $p_{\mathfrak{X}}$ is smooth at this point (see eg. \cite[Proposition V.3.8]{ska} or \cite[Théorème 2.1]{DM98}). 
\end{proof}

For $(n,d,r)=(5,2,2)$, where $\delta(5,2,2)=0$, a complete proof of the non-emptiness of $F_2(X)$ for every smooth $X$ is given in Remark~\ref{rmk:f-planes} (see Proposition~\ref{prop:planes-incidence}), which we expect that can be extended to other situations.

\begin{remark}\label{remark:Fano-scheme}
\begin{enumerate}
    \item The Fano scheme $F_r(X)$ is the zero locus of a section of $\mathcal{O}_{G_P(r)}(2) \otimes \gamma^* \mathrm{Sym}^{2d}(Q)$. Thus
    \[
    [F_r(X)] = c_{\binom{2d+r}{2d}}\!\left( \mathcal{O}_{G_P(r)}(2) \otimes \gamma^* \mathrm{Sym}^{2d}(Q) \right) \in \CH^*(G_P(r)).
    \]
    \item If $X \xrightarrow{m:1} \mathbb{P}^n$ is a cyclic covering of order $m$ ramified along a hypersurface of degree $d$, then $X$ can be described as the zero locus of a section of $\mathcal{O}_P(m)$. In this case, the expected dimension formula generalizes to
    \[
    \dim F_r(X) = (r+1)(n-r) + \binom{d+r}{d} - \binom{md+r}{md}.
    \]
\end{enumerate}
\end{remark}

\begin{prop}
Let $X \xrightarrow{2:1} \mathbb{P}^n$ be a smooth double covering branched along a hypersurface $B \subset \mathbb{P}^n$ of degree $2d$. Fix $r \geq 1$. Assume $F_r(X)$ is smooth, then the canonical bundle of $F_r(X)$ is
\[
\omega_{F_r(X)} \simeq \left( \gamma^*\mathcal{O}_{\mathbb{G}(r,n)}(a+b-n-1) \otimes \mathcal{O}_{G_P(r)}\!\left( 2\binom{2d+r}{2d} - \binom{d+r}{d} - 1 \right) \right)\big|_{F_r(X)},
\]
where
\[
a = \begin{cases}
\frac{d(d+1)}{2}, & r=1, \\
\frac{r^{d+1} - r(d+1) + d}{(r-1)^2}, & r \neq 1,
\end{cases}
\quad
b = \begin{cases}
d(2d+1), & r=1, \\
\frac{r^{2d+1} - r(2d+1) + 2d}{(r-1)^2}, & r \neq 1.
\end{cases}
\]
\end{prop}

\begin{proof}
By Remark~\ref{remark:Fano-scheme}, there exists $s \in H^0(G_P(r), \mathcal{O}_{G_P(r)}(2) \otimes \gamma^*\mathrm{Sym}^{2d}(Q))$ with $F_r(X) \simeq V(s)$. Thus
\[
\omega_{F_r(X)} \simeq \left( \omega_{G_P(r)} \otimes \det \mathcal{F} \right)\big|_{F_r(X)},
\]
where $\mathcal{F} := \mathcal{O}_{G_P(r)}(2) \otimes \gamma^*\mathrm{Sym}^{2d}(Q)$.

The morphism $\gamma: G_P(r) \to \mathbb{G}(r,n)$ is the $\mathbb{P}^{\binom{d+r}{d}}$-bundle obtained as the projectivization of $\mathcal{E} := \mathcal{O}_{\mathbb{G}(r,n)} \oplus \mathrm{Sym}^d(Q^\vee)$. Using the splitting principle, one checks
\[
\det(\mathcal{E})\simeq \det(\mathrm{Sym}^d(Q^\vee)) \simeq \mathcal{O}_{\mathbb{G}(r,n)}(a).
\]
Since $\omega_{\mathbb{G}(r,n)} \simeq \mathcal{O}_{\mathbb{G}(r,n)}(-n-1)$, we have
\[
\omega_{G_P(r)} = \gamma^*\mathcal{O}_{\mathbb{G}(r,n)}(a-n-1) \otimes \mathcal{O}_{G_P(r)}\!\left(-\binom{d+r}{d} - 1\right).
\]

Similarly,
\[
\det(\mathcal{F}) \simeq \mathcal{O}_{G_P(r)}\!\left(2\binom{2d+r}{2d}\right) \otimes \gamma^*\mathcal{O}_{\mathbb{G}(r,n)}(b).
\]
Combining, we obtain
\[
\omega_{F_r(X)} \simeq \left( \gamma^*\mathcal{O}_{\mathbb{G}(r,n)}(a+b-n-1) \otimes \mathcal{O}_{G_P(r)}\!\left( 2\binom{2d+r}{2d} - \binom{d+r}{d} - 1 \right) \right)\big|_{F_r(X)},
\]
as long as $F_r(X)$ is smooth.
\end{proof}

\printbibliography[title={Bibliography}]

@article {Gei17,
    AUTHOR = {Geisser, Thomas H.},
     TITLE = {On the structure of \'etale motivic cohomology},
   JOURNAL = {J. Pure Appl. Algebra},
  FJOURNAL = {Journal of Pure and Applied Algebra},
    VOLUME = {221},
      YEAR = {2017},
    NUMBER = {7},
     PAGES = {1614--1628},
      ISSN = {0022-4049,1873-1376},
   MRCLASS = {14F42 (14C15)},
  MRNUMBER = {3614969},
MRREVIEWER = {Florence\ Lecomte},
       DOI = {10.1016/j.jpaa.2016.12.019},
       URL = {https://doi.org/10.1016/j.jpaa.2016.12.019},
}

@book {Fuj90,
    AUTHOR = {Fujita, Takao},
     TITLE = {Classification theories of polarized varieties},
    SERIES = {London Mathematical Society Lecture Note Series},
    VOLUME = {155},
 PUBLISHER = {Cambridge University Press, Cambridge},
      YEAR = {1990},
     PAGES = {xiv+205},
      ISBN = {0-521-39202-0},
   MRCLASS = {14C20 (14J40 14J60)},
  MRNUMBER = {1162108},
MRREVIEWER = {Elvira\ Laura\ Livorni},
       DOI = {10.1017/CBO9780511662638},
       URL = {https://doi.org/10.1017/CBO9780511662638},
}

@incollection {Voe00,
    AUTHOR = {Voevodsky, Vladimir},
     TITLE = {Triangulated categories of motives over a field},
 BOOKTITLE = {Cycles, transfers, and motivic homology theories},
    SERIES = {Ann. of Math. Stud.},
    VOLUME = {143},
     PAGES = {188--238},
 PUBLISHER = {Princeton Univ. Press, Princeton, NJ},
      YEAR = {2000},
      ISBN = {0-691-04814-2; 0-691-04815-0},
   MRCLASS = {14F42 (14C25)},
  MRNUMBER = {1764202},
}

@article {Sch23,
    AUTHOR = {Schreieder, Stefan},
     TITLE = {Refined unramified cohomology of schemes},
   JOURNAL = {Compos. Math.},
  FJOURNAL = {Compositio Mathematica},
    VOLUME = {159},
      YEAR = {2023},
    NUMBER = {7},
     PAGES = {1466--1530},
      ISSN = {0010-437X,1570-5846},
   MRCLASS = {14C25 (14C30 14F20)},
  MRNUMBER = {4603634},
MRREVIEWER = {Fumio\ Hazama},
       DOI = {10.1112/s0010437x23007236},
       URL = {https://doi.org/10.1112/s0010437x23007236},
}

@article{AS24,
  title={Truncated pushforwards and refined unramified cohomology},
  author={Alexandrou, Theodosis and Schreieder, Stefan},
  journal={Advances in Mathematics},
  volume={458},
  pages={109979},
  year={2024},
  publisher={Elsevier}
}

@article {CTP16,
    AUTHOR = {Colliot-Th\'el\`ene, Jean-Louis and Pirutka, Alena},
     TITLE = {Hypersurfaces quartiques de dimension 3: non-rationalit\'e{}
              stable},
   JOURNAL = {Ann. Sci. \'Ec. Norm. Sup\'er. (4)},
  FJOURNAL = {Annales Scientifiques de l'\'Ecole Normale Sup\'erieure.
              Quatri\`eme S\'erie},
    VOLUME = {49},
      YEAR = {2016},
    NUMBER = {2},
     PAGES = {371--397},
      ISSN = {0012-9593,1873-2151},
   MRCLASS = {14G25 (14M20)},
  MRNUMBER = {3481353},
MRREVIEWER = {Abhishek\ Banerjee},
       DOI = {10.24033/asens.2285},
       URL = {https://doi.org/10.24033/asens.2285},
}

@incollection {Pey19,
    AUTHOR = {Peyre, Emmanuel},
     TITLE = {Progr\`es en irrationalit\'e{} [d'apr\`es {C}. {V}oisin,
              {J}.-{L}. {C}olliot-{T}h\'el\`ene, {B}. {H}assett, {A}.
              {K}resch, {A}. {P}irutka, {B}. {T}otaro, {Y}. {T}schinkel et
              al.]},
      NOTE = {S\'eminaire Bourbaki. Vol. 2016/2017. Expos\'es 1120--1135},
   JOURNAL = {Ast\'erisque},
  FJOURNAL = {Ast\'erisque},
    NUMBER = {407},
      YEAR = {2019},
     PAGES = {Exp. No. 1123, 91--116},
      ISSN = {0303-1179,2492-5926},
      ISBN = {978-2-85629-897-8},
   MRCLASS = {14E08},
  MRNUMBER = {3939274},
MRREVIEWER = {Anne-Sophie\ Kaloghiros},
       DOI = {10.24033/ast},
       URL = {https://doi.org/10.24033/ast},
}

@incollection {Voi19,
    AUTHOR = {Voisin, Claire},
     TITLE = {Birational invariants and decomposition of the diagonal},
 BOOKTITLE = {Birational geometry of hypersurfaces},
    SERIES = {Lect. Notes Unione Mat. Ital.},
    VOLUME = {26},
     PAGES = {3--71},
 PUBLISHER = {Springer, Cham},
      YEAR = {2019},
      ISBN = {978-3-030-18637-1; 978-3-030-18638-8},
   MRCLASS = {14C15 (14E08)},
  MRNUMBER = {4401008},
MRREVIEWER = {Alex\ Massarenti},
       DOI = {10.1007/978-3-030-18638-8\_1},
       URL = {https://doi.org/10.1007/978-3-030-18638-8_1},
}

@article {CTO89,
    AUTHOR = {Colliot-Th\'el\`ene, Jean-Louis and Ojanguren, Manuel},
     TITLE = {Vari\'et\'es unirationnelles non rationnelles: au-del\`a{} de
              l'exemple d'{A}rtin et {M}umford},
   JOURNAL = {Invent. Math.},
  FJOURNAL = {Inventiones Mathematicae},
    VOLUME = {97},
      YEAR = {1989},
    NUMBER = {1},
     PAGES = {141--158},
      ISSN = {0020-9910,1432-1297},
   MRCLASS = {14E05 (11E76 11G30 14F99 14J10 14M99)},
  MRNUMBER = {999316},
MRREVIEWER = {Wayne\ Raskind},
       DOI = {10.1007/BF01850658},
       URL = {https://doi.org/10.1007/BF01850658},
}

@incollection {Mur04,
    AUTHOR = {Murre, Jacob P.},
     TITLE = {Fano varieties and algebraic cycles},
 BOOKTITLE = {The {F}ano {C}onference},
     PAGES = {51--68},
 PUBLISHER = {Univ. Torino, Turin},
      YEAR = {2004},
      ISBN = {88-900876-1-7},
   MRCLASS = {14J45 (14C25)},
  MRNUMBER = {2112567},
MRREVIEWER = {Baohua\ Fu},
}

@article {BO74,
    AUTHOR = {Bloch, Spencer and Ogus, Arthur},
     TITLE = {Gersten's conjecture and the homology of schemes},
   JOURNAL = {Ann. Sci. \'Ecole Norm. Sup. (4)},
  FJOURNAL = {Annales Scientifiques de l'\'Ecole Normale Sup\'erieure.
              Quatri\`eme S\'erie},
    VOLUME = {7},
      YEAR = {1974},
     PAGES = {181--201},
      ISSN = {0012-9593},
   MRCLASS = {14F15 (14C10 14F20)},
  MRNUMBER = {412191},
MRREVIEWER = {J.-P.\ Jouanolou},
       URL = {http://www.numdam.org/item?id=ASENS_1974_4_7_2_181_0},
}

@article {Bea09,
    AUTHOR = {Beauville, Arnaud},
     TITLE = {On the {B}rauer group of {E}nriques surfaces},
   JOURNAL = {Math. Res. Lett.},
  FJOURNAL = {Mathematical Research Letters},
    VOLUME = {16},
      YEAR = {2009},
    NUMBER = {6},
     PAGES = {927--934},
      ISSN = {1073-2780},
   MRCLASS = {14J28 (14F22)},
  MRNUMBER = {2576681},
MRREVIEWER = {Stefan\ Schr\"oer},
       DOI = {10.4310/MRL.2009.v16.n6.a1},
       URL = {https://doi.org/10.4310/MRL.2009.v16.n6.a1},
}

@book{ska,
  title={Rationality Properties of Some Hypersurfaces and Complete Intersections},
  author={Skauli, Bj{\o}rn},
  year={2023},
  publisher={University of Oslo},
    journal={PhD thesis}
}

@article {Tih,
    AUTHOR = {Tihomirov, A. S.},
     TITLE = {Geometry of the {F}ano surface of a double {${\bf P}\sp{3}$}\
              branched in a quartic},
   JOURNAL = {Izv. Akad. Nauk SSSR Ser. Mat.},
  FJOURNAL = {Izvestiya Akademii Nauk SSSR. Seriya Matematicheskaya},
    VOLUME = {44},
      YEAR = {1980},
    NUMBER = {2},
     PAGES = {415--442, 479},
      ISSN = {0373-2436},
   MRCLASS = {14J25 (14J30 14K30)},
  MRNUMBER = {571103},
MRREVIEWER = {Werner\ Kleinert},
}

@book{3264,
    AUTHOR = {Eisenbud, David and Harris, Joe},
     TITLE = {3264 and all that---a second course in algebraic geometry},
 PUBLISHER = {Cambridge University Press, Cambridge},
      YEAR = {2016},
     PAGES = {xiv+616},
      ISBN = {978-1-107-60272-4},
   MRCLASS = {14-01 (14C15 14M15 14N10)},
  MRNUMBER = {3617981},
MRREVIEWER = {Arnaud\ Beauville},
       DOI = {10.1017/CBO9781139062046},
       URL = {https://doi.org/10.1017/CBO9781139062046},
}

@article {Suz,
    AUTHOR = {Suzuki, Fumiaki},
     TITLE = {A remark on a 3-fold constructed by {C}olliot-{T}h\'el\`ene
              and {V}oisin},
   JOURNAL = {Math. Res. Lett.},
  FJOURNAL = {Mathematical Research Letters},
    VOLUME = {27},
      YEAR = {2020},
    NUMBER = {1},
     PAGES = {301--317},
      ISSN = {1073-2780,1945-001X},
   MRCLASS = {14C15 (14J35)},
  MRNUMBER = {4088818},
MRREVIEWER = {Pavel\ Sechin},
       DOI = {10.4310/mrl.2020.v27.n1.a15},
       URL = {https://doi.org/10.4310/mrl.2020.v27.n1.a15},
}

@misc{RoSo,
      title={Hodge structures through an \'etale motivic point of view},
      author={Rosas Soto, Ivan},
      year={2022},
      eprint={2212.02128},
      archivePrefix={arXiv},
      primaryClass={math.AG}
}

@article{sus,
  title={Higher Chow groups and etale cohomology},
  author={Suslin, Andrei},
  journal={Cycles, transfers, and motivic homology theories},
  volume={143},
  pages={239--254},
  year={2000},
  publisher={Princeton Univ. Press Princeton, NJ}
}

@book {CD19,
    AUTHOR = {Cisinski, Denis-Charles and D\'{e}glise, Fr\'{e}d\'{e}ric},
     TITLE = {Triangulated categories of mixed motives},
    SERIES = {Springer Monographs in Mathematics},
 PUBLISHER = {Springer, Cham},
      YEAR = {2019},
     PAGES = {xlii+406},
      ISBN = {978-3-030-33241-9},
   MRCLASS = {14F42 (14C15 14C35 18G80 19D55)},
  MRNUMBER = {3971240},
MRREVIEWER = {Igor A. Rapinchuk},
       DOI = {10.1007/978-3-030-33242-6},
       URL = {https://doi.org/10.1007/978-3-030-33242-6},
}

@article {RS,
    AUTHOR = {Rosenschon, Andreas and Srinivas, V.},
     TITLE = {\'{E}tale motivic cohomology and algebraic cycles},
   JOURNAL = {J. Inst. Math. Jussieu},
  FJOURNAL = {Journal of the Institute of Mathematics of Jussieu. JIMJ.
              Journal de l'Institut de Math\'{e}matiques de Jussieu},
    VOLUME = {15},
      YEAR = {2016},
    NUMBER = {3},
     PAGES = {511--537},
      ISSN = {1474-7480},
   MRCLASS = {14F42 (14F20 19E15 19E20)},
  MRNUMBER = {3505657},
MRREVIEWER = {Florence Lecomte},
       DOI = {10.1017/S1474748014000401},
       URL = {https://doi.org/10.1017/S1474748014000401},
}

@article{DM98,
  title={Sur la vari\'et\'e des espaces lin\'eaires contenus dans une intersection complète},
  author={Debarre, Olivier and Manivel, Laurent},
  journal={Math. Ann},
  volume={312},
  pages={549--574},
  year={1998}
}

@article{BS83,
  title={Remarks on correspondences and algebraic cycles},
  author={Bloch, Spencer and Srinivas, Vasudevan},
  journal={American Journal of Mathematics},
  volume={105},
  number={5},
  pages={1235--1253},
  year={1983},
  publisher={JSTOR}
}

@article {AH,
    AUTHOR = {Atiyah, M. F. and Hirzebruch, F.},
     TITLE = {Analytic cycles on complex manifolds},
   JOURNAL = {Topology},
  FJOURNAL = {Topology. An International Journal of Mathematics},
    VOLUME = {1},
      YEAR = {1962},
     PAGES = {25--45},
      ISSN = {0040-9383},
   MRCLASS = {57.60 (57.32)},
  MRNUMBER = {145560},
MRREVIEWER = {A. Haefliger},
       DOI = {10.1016/0040-9383(62)90094-0},
       URL = {https://doi.org/10.1016/0040-9383(62)90094-0},
}

@article {Ayo,
    AUTHOR = {Ayoub, Joseph},
     TITLE = {La r\'{e}alisation \'{e}tale et les op\'{e}rations de {G}rothendieck},
   JOURNAL = {Ann. Sci. \'{E}c. Norm. Sup\'{e}r. (4)},
  FJOURNAL = {Annales Scientifiques de l'\'{E}cole Normale Sup\'{e}rieure. Quatri\`eme
              S\'{e}rie},
    VOLUME = {47},
      YEAR = {2014},
    NUMBER = {1},
     PAGES = {1--145},
      ISSN = {0012-9593},
   MRCLASS = {14C15 (14F05 14F20 14F42 18G55)},
  MRNUMBER = {3205601},
MRREVIEWER = {C. A. M. Peters},
       DOI = {10.24033/asens.2210},
       URL = {https://doi.org/10.24033/asens.2210},
}

@proceedings {BCC,
     TITLE = {Classification of irregular varieties},
    SERIES = {Lecture Notes in Mathematics},
    VOLUME = {1515},
 BOOKTITLE = {Proceedings of the conference held in {T}rento, {D}ecember
              17--21, 1990},
    EDITOR = {Ballico, E. and Catanese, F. and Ciliberto, C.},
      NOTE = {Minimal models and abelian varieties},
 PUBLISHER = {Springer-Verlag, Berlin},
      YEAR = {1992},
     PAGES = {vi+149},
      ISBN = {3-540-55295-2},
   MRCLASS = {14-06},
  MRNUMBER = {1180332},
       DOI = {10.1007/BFb0098332},
       URL = {https://doi.org/10.1007/BFb0098332},
}

@article {Blo,
    AUTHOR = {Bloch, Spencer},
     TITLE = {Algebraic cycles and higher {$K$}-theory},
   JOURNAL = {Adv. in Math.},
  FJOURNAL = {Advances in Mathematics},
    VOLUME = {61},
      YEAR = {1986},
    NUMBER = {3},
     PAGES = {267--304},
      ISSN = {0001-8708},
   MRCLASS = {18F25 (11G45 14C35 19D99 19E15)},
  MRNUMBER = {852815},
MRREVIEWER = {L. N. Vaserstein},
       DOI = {10.1016/0001-8708(86)90081-2},
       URL = {https://doi.org/10.1016/0001-8708(86)90081-2},
}

@article {CTV,
    AUTHOR = {Colliot-Th\'{e}l\`ene, Jean-Louis and Voisin, Claire},
     TITLE = {Cohomologie non ramifi\'{e}e et conjecture de {H}odge enti\`ere},
   JOURNAL = {Duke Math. J.},
  FJOURNAL = {Duke Mathematical Journal},
    VOLUME = {161},
      YEAR = {2012},
    NUMBER = {5},
     PAGES = {735--801},
      ISSN = {0012-7094},
   MRCLASS = {14F43 (14C30 14E05)},
  MRNUMBER = {2904092},
MRREVIEWER = {Christian Schnell},
       DOI = {10.1215/00127094-1548389},
       URL = {https://doi.org/10.1215/00127094-1548389},
}

@article {CD16,
    AUTHOR = {Cisinski, Denis-Charles and D\'{e}glise, Fr\'{e}d\'{e}ric},
     TITLE = {\'{E}tale motives},
   JOURNAL = {Compos. Math.},
  FJOURNAL = {Compositio Mathematica},
    VOLUME = {152},
      YEAR = {2016},
    NUMBER = {3},
     PAGES = {556--666},
      ISSN = {0010-437X},
   MRCLASS = {14F20 (14F42)},
  MRNUMBER = {3477640},
MRREVIEWER = {Matthias Wendt},
       DOI = {10.1112/S0010437X15007459},
       URL = {https://doi.org/10.1112/S0010437X15007459},
}

@book {Fulton,
    AUTHOR = {Fulton, William},
     TITLE = {Intersection theory},
    SERIES = {Ergebnisse der Mathematik und ihrer Grenzgebiete. 3. Folge. A
              Series of Modern Surveys in Mathematics [Results in
              Mathematics and Related Areas. 3rd Series. A Series of Modern
              Surveys in Mathematics]},
    VOLUME = {2},
   EDITION = {Second},
 PUBLISHER = {Springer-Verlag, Berlin},
      YEAR = {1998},
     PAGES = {xiv+470},
      ISBN = {0-387-98549-2},
   MRCLASS = {14C17 (14-02)},
  MRNUMBER = {1644323},
       DOI = {10.1007/978-1-4612-1700-8},
       URL = {https://doi.org/10.1007/978-1-4612-1700-8},
}

@article {Ge04,
    AUTHOR = {Geisser, Thomas},
     TITLE = {Motivic cohomology over {D}edekind rings},
   JOURNAL = {Math. Z.},
  FJOURNAL = {Mathematische Zeitschrift},
    VOLUME = {248},
      YEAR = {2004},
    NUMBER = {4},
     PAGES = {773--794},
      ISSN = {0025-5874},
   MRCLASS = {14F42 (14G40)},
  MRNUMBER = {2103541},
MRREVIEWER = {Vladimir I. Guletski\u{\i}},
       DOI = {10.1007/s00209-004-0680-x},
       URL = {https://doi.org/10.1007/s00209-004-0680-x},
}

@inproceedings {Geis,
    AUTHOR = {Geisser, Thomas H.},
     TITLE = {Duality of integral \'{e}tale motivic cohomology},
 BOOKTITLE = {{$K$}-{T}heory---{P}roceedings of the {I}nternational
              {C}olloquium, {M}umbai, 2016},
     PAGES = {195--209},
 PUBLISHER = {Hindustan Book Agency, New Delhi},
      YEAR = {2018},
   MRCLASS = {19E15},
  MRNUMBER = {3930049},
MRREVIEWER = {Matthias Wendt},
}

@article {Kahn,
    AUTHOR = {Kahn, Bruno},
     TITLE = {Classes de cycles motiviques \'{e}tales},
   JOURNAL = {Algebra Number Theory},
  FJOURNAL = {Algebra \& Number Theory},
    VOLUME = {6},
      YEAR = {2012},
    NUMBER = {7},
     PAGES = {1369--1407},
      ISSN = {1937-0652},
   MRCLASS = {14C25 (14F20 14F42 14G17)},
  MRNUMBER = {3007153},
MRREVIEWER = {Claudio Pedrini},
       DOI = {10.2140/ant.2012.6.1369},
       URL = {https://doi.org/10.2140/ant.2012.6.1369},
}

@book {Lew,
    AUTHOR = {Lewis, James D.},
     TITLE = {A survey of the {H}odge conjecture},
    SERIES = {CRM Monograph Series},
    VOLUME = {10},
   EDITION = {Second},
      NOTE = {Appendix B by B. Brent Gordon},
 PUBLISHER = {American Mathematical Society, Providence, RI},
      YEAR = {1999},
     PAGES = {xvi+368},
      ISBN = {0-8218-0568-1},
   MRCLASS = {14C30 (32J25)},
  MRNUMBER = {1683216},
       DOI = {10.1090/crmm/010},
       URL = {https://doi.org/10.1090/crmm/010},
}

@book {Mil80,
    AUTHOR = {Milne, James S.},
     TITLE = {\'{E}tale cohomology},
    SERIES = {Princeton Mathematical Series, No. 33},
 PUBLISHER = {Princeton University Press, Princeton, N.J.},
      YEAR = {1980},
     PAGES = {xiii+323},
      ISBN = {0-691-08238-3},
   MRCLASS = {14-02 (14F20 18F99)},
  MRNUMBER = {559531},
MRREVIEWER = {G. Horrocks},
}

@book {MVW,
    AUTHOR = {Mazza, Carlo and Voevodsky, Vladimir and Weibel, Charles},
     TITLE = {Lecture notes on motivic cohomology},
    SERIES = {Clay Mathematics Monographs},
    VOLUME = {2},
 PUBLISHER = {American Mathematical Society, Providence, RI; Clay
              Mathematics Institute, Cambridge, MA},
      YEAR = {2006},
     PAGES = {xiv+216},
      ISBN = {978-0-8218-3847-1},
   MRCLASS = {14F42 (19E15)},
  MRNUMBER = {2242284},
MRREVIEWER = {Thomas Geisser},
}

@article {VV,
    AUTHOR = {Voevodsky, Vladimir},
     TITLE = {On motivic cohomology with {$\mathbf{Z}/l$}-coefficients},
   JOURNAL = {Ann. of Math. (2)},
  FJOURNAL = {Annals of Mathematics. Second Series},
    VOLUME = {174},
      YEAR = {2011},
    NUMBER = {1},
     PAGES = {401--438},
      ISSN = {0003-486X},
   MRCLASS = {14F42 (19D45)},
  MRNUMBER = {2811603},
MRREVIEWER = {Matthias Wendt},
       DOI = {10.4007/annals.2011.174.1.11},
       URL = {https://doi.org/10.4007/annals.2011.174.1.11},
}

@book {V,
    AUTHOR = {Voevodsky, Vladimir and Suslin, Andrei and Friedlander, Eric
              M.},
     TITLE = {Cycles, transfers, and motivic homology theories},
    SERIES = {Annals of Mathematics Studies},
    VOLUME = {143},
 PUBLISHER = {Princeton University Press, Princeton, NJ},
      YEAR = {2000},
     PAGES = {vi+254},
      ISBN = {0-691-04814-2},
   MRCLASS = {14F42 (14C25 19E15)},
  MRNUMBER = {1764197},
MRREVIEWER = {Spencer J. Bloch},
}

@book {Voi,
    AUTHOR = {Voisin, Claire},
     TITLE = {Th\'{e}orie de {H}odge et g\'{e}om\'{e}trie alg\'{e}brique complexe},
    SERIES = {Cours Sp\'{e}cialis\'{e}s [Specialized Courses]},
    VOLUME = {10},
 PUBLISHER = {Soci\'{e}t\'{e} Math\'{e}matique de France, Paris},
      YEAR = {2002},
     PAGES = {viii+595},
      ISBN = {2-85629-129-5},
   MRCLASS = {32J25 (14C30 14D07 32G20)},
  MRNUMBER = {1988456},
MRREVIEWER = {Javier A. Fern\'{a}ndez},
       DOI = {10.1017/CBO9780511615344},
       URL = {https://doi.org/10.1017/CBO9780511615344},
}

@book {EV,
    AUTHOR = {Esnault, H\'el\`ene and Viehweg, Eckart},
     TITLE = {Lectures on vanishing theorems},
    SERIES = {DMV Seminar},
    VOLUME = {20},
 PUBLISHER = {Birkh\"auser Verlag, Basel},
      YEAR = {1992},
     PAGES = {vi+164},
      ISBN = {3-7643-2822-3},
   MRCLASS = {14F17 (14F40 32L10 32L20)},
  MRNUMBER = {1193913},
MRREVIEWER = {Marko\ Roczen},
       DOI = {10.1007/978-3-0348-8600-0},
       URL = {https://doi.org/10.1007/978-3-0348-8600-0},
}

@incollection {Murr,
    AUTHOR = {Murre, J. P.},
     TITLE = {Applications of algebraic {$K$}-theory to the theory of
              algebraic cycles},
 BOOKTITLE = {Algebraic geometry, {S}itges ({B}arcelona), 1983},
    SERIES = {Lecture Notes in Math.},
    VOLUME = {1124},
     PAGES = {216--261},
 PUBLISHER = {Springer, Berlin},
      YEAR = {1985},
      ISBN = {3-540-15232-6},
   MRCLASS = {14C35 (14C10 14K30 18F25 19E15)},
  MRNUMBER = {805336},
MRREVIEWER = {K.\ R.\ Coombes},
       DOI = {10.1007/BFb0075002},
       URL = {https://doi.org/10.1007/BFb0075002},
}

@article {Bea15,
    AUTHOR = {Beauville, Arnaud},
     TITLE = {A very general quartic double fourfold or fivefold is not
              stably rational},
   JOURNAL = {Algebr. Geom.},
  FJOURNAL = {Algebraic Geometry},
    VOLUME = {2},
      YEAR = {2015},
    NUMBER = {4},
     PAGES = {508--513},
      ISSN = {2313-1691,2214-2584},
   MRCLASS = {14E20 (14E08)},
  MRNUMBER = {3403239},
MRREVIEWER = {Alberto\ Calabri},
       DOI = {10.14231/AG-2015-022},
       URL = {https://doi.org/10.14231/AG-2015-022},
}

@incollection {Sch21,
    AUTHOR = {Schreieder, Stefan},
     TITLE = {Unramified cohomology, algebraic cycles and rationality},
 BOOKTITLE = {Rationality of varieties},
    EDITOR = {Farkas, Gavril and van der Geer, Gerard and Shen, Mingmin and Taelman, Lenny},
    SERIES = {Progr. Math.},
    VOLUME = {342},
     PAGES = {345--388},
 PUBLISHER = {Birkh\"auser/Springer, Cham},
      YEAR = {2021},
       DOI = {10.1007/978-3-030-75421-1_13},
      NOTE = {arXiv:2106.01057},
}

@misc{Cha26,
      title={A quadratically enriched count of lines in smooth del {P}ezzo surfaces of degree 2 and 4},
      author={Chachay, Victor},
      year={2026},
      eprint={2603.16315},
      archivePrefix={arXiv},
      primaryClass={math.AG},
      url={https://arxiv.org/abs/2603.16315},
}

@book {KSC04,
    AUTHOR = {Koll\'ar, J\'anos and Smith, Karen E. and Corti, Alessio},
     TITLE = {Rational and nearly rational varieties},
    SERIES = {Cambridge Studies in Advanced Mathematics},
    VOLUME = {92},
 PUBLISHER = {Cambridge University Press, Cambridge},
      YEAR = {2004},
     PAGES = {vi+235},
      ISBN = {0-521-83207-1},
   MRCLASS = {14E08 (14-02 14J26 14J45 14M20)},
  MRNUMBER = {2062787},
       DOI = {10.1017/CBO9780511734991},
}

@book {Kol96,
    AUTHOR = {Koll\'ar, J\'anos},
     TITLE = {Rational curves on algebraic varieties},
    SERIES = {Ergebnisse der Mathematik und ihrer Grenzgebiete. 3. Folge},
    VOLUME = {32},
 PUBLISHER = {Springer-Verlag, Berlin},
      YEAR = {1996},
     PAGES = {viii+320},
      ISBN = {3-540-60168-6},
   MRCLASS = {14-02 (14E05 14H10 14J45 14M20)},
  MRNUMBER = {1440180},
       DOI = {10.1007/978-3-662-03276-3},
}

@article {IM15,
    AUTHOR = {Iliev, Atanas and Manivel, Laurent},
     TITLE = {Fano manifolds of {C}alabi--{Y}au {H}odge type},
   JOURNAL = {J. Pure Appl. Algebra},
  FJOURNAL = {Journal of Pure and Applied Algebra},
    VOLUME = {219},
      YEAR = {2015},
    NUMBER = {6},
     PAGES = {2225--2244},
}

@article {FIK14,
    AUTHOR = {Favero, David and Iliev, Atanas and Katzarkov, Ludmil},
     TITLE = {On the {G}riffiths groups of {F}ano manifolds of {C}alabi--{Y}au {H}odge type},
   JOURNAL = {Pure Appl. Math. Q.},
  FJOURNAL = {Pure and Applied Mathematics Quarterly},
    VOLUME = {10},
      YEAR = {2014},
    NUMBER = {1},
     PAGES = {1--55},
      NOTE = {arXiv:1212.2608},
}

@article {Abu20,
    AUTHOR = {Abuaf, Roland},
     TITLE = {On quartic double fivefolds and the matrix factorizations of exceptional quaternionic representations},
   JOURNAL = {Ann. Inst. Fourier (Grenoble)},
  FJOURNAL = {Universit\'e{} de Grenoble. Annales de l'Institut Fourier},
    VOLUME = {70},
      YEAR = {2020},
    NUMBER = {4},
     PAGES = {1403--1430},
       DOI = {10.5802/aif.3396},
}

@article {Ma22,
    AUTHOR = {Ma, Shouhei},
     TITLE = {Unramified cohomology, integral coniveau filtration and {G}riffiths groups},
   JOURNAL = {Ann. K-Theory},
  FJOURNAL = {Annals of K-Theory},
    VOLUME = {7},
      YEAR = {2022},
    NUMBER = {2},
     PAGES = {223--236},
       DOI = {10.2140/akt.2022.7.223},
}

@article {Isk96,
    AUTHOR = {Iskovskikh, V. A.},
     TITLE = {Factorization of birational maps of rational surfaces from the viewpoint of {M}ori theory},
   JOURNAL = {Uspekhi Mat. Nauk},
  FJOURNAL = {Uspekhi Matematicheskikh Nauk},
    VOLUME = {51},
      YEAR = {1996},
    NUMBER = {4(310)},
     PAGES = {3--72},
      NOTE = {English translation in Russian Math. Surveys \textbf{51} (1996), no.~4, 585--652},
   MRNUMBER = {1422227},
}

@article {Mer08,
    AUTHOR = {Merkurjev, Alexander},
     TITLE = {Unramified elements in cycle modules},
   JOURNAL = {J. Lond. Math. Soc. (2)},
  FJOURNAL = {Journal of the London Mathematical Society. Second Series},
    VOLUME = {78},
      YEAR = {2008},
    NUMBER = {1},
     PAGES = {51--64},
       DOI = {10.1112/jlms/jdn011},
}

@article {Cor81,
    AUTHOR = {Cornalba, Maurizio},
     TITLE = {Una osservazione sulla topologia dei rivestimenti ciclici di variet\`a{} algebriche},
   JOURNAL = {Boll. Un. Mat. Ital. A (5)},
  FJOURNAL = {Bollettino della Unione Matematica Italiana. Serie V. A},
    VOLUME = {18},
      YEAR = {1981},
     PAGES = {323--328},
}

@incollection {Dol82,
    AUTHOR = {Dolgachev, Igor},
     TITLE = {Weighted projective varieties},
 BOOKTITLE = {Group actions and vector fields ({V}ancouver, {B}.{C}., 1981)},
    SERIES = {Lecture Notes in Math.},
    VOLUME = {956},
     PAGES = {34--71},
 PUBLISHER = {Springer, Berlin},
      YEAR = {1982},
   MRNUMBER = {704986},
       DOI = {10.1007/BFb0101508},
}

@book {MFK94,
    AUTHOR = {Mumford, D. and Fogarty, J. and Kirwan, F.},
     TITLE = {Geometric invariant theory},
    SERIES = {Ergebnisse der Mathematik und ihrer Grenzgebiete (2)},
    VOLUME = {34},
   EDITION = {Third},
 PUBLISHER = {Springer-Verlag, Berlin},
      YEAR = {1994},
     PAGES = {xiv+292},
      ISBN = {3-540-56963-4},
   MRNUMBER = {1304906},
       DOI = {10.1007/978-3-642-57916-5},
}

@book {Ser02,
    AUTHOR = {Serre, Jean-Pierre},
     TITLE = {Galois cohomology},
    SERIES = {Springer Monographs in Mathematics},
      NOTE = {Translated from the French by Patrick Ion and revised by the author. Corrected reprint of the 1997 English edition},
 PUBLISHER = {Springer-Verlag, Berlin},
      YEAR = {2002},
     PAGES = {x+210},
      ISBN = {3-540-42192-0},
   MRNUMBER = {1867431},
}

@article {Bea16a,
    AUTHOR = {Beauville, Arnaud},
     TITLE = {Erratum: ``{A} very general quartic double fourfold or fivefold is not stably rational''},
   JOURNAL = {Algebr. Geom.},
  FJOURNAL = {Algebraic Geometry},
    VOLUME = {3},
      YEAR = {2016},
    NUMBER = {1},
     PAGES = {137},
       DOI = {10.14231/AG-2016-006},
}

@misc{MP17,
      title={1-{C}ycles on {F}ano varieties},
      author={Minoccheri, Cristian and Pan, Xuanyu},
      year={2017},
      eprint={1711.09987},
      archivePrefix={arXiv},
      primaryClass={math.AG},
}

@article {TZ14,
    AUTHOR = {Tian, Zhiyu and Zong, Hong R.},
     TITLE = {One-cycles on rationally connected varieties},
   JOURNAL = {Compos. Math.},
  FJOURNAL = {Compositio Mathematica},
    VOLUME = {150},
      YEAR = {2014},
    NUMBER = {3},
     PAGES = {396--408},
       DOI = {10.1112/S0010437X13007549},
}

@misc{MRII,
      title={Algebraic cycles on some {F}ano varieties. {P}art {II}: {H}odge structures of level one},
      author={Montero, Pedro and Rosas-Soto, Iv\'an},
      year={2026},
      note={Preprint},
}

\info
\end{document}